\documentclass[12pt]{amsart} 

\usepackage[english]{babel}
\usepackage[latin1]{inputenc} 
\usepackage{amsmath}
\usepackage{amsthm}
\usepackage{amssymb}
\usepackage{csquotes}
\usepackage{tikz-cd}

\usepackage{enumerate}
\usepackage{imakeidx}
\usepackage{appendix}
\usepackage{tikz-cd}
\usepackage{verbatim}
\usepackage{enumitem}
\usepackage{multicol}
\usepackage[normalem]{ulem}

\usepackage[backref=page]{hyperref}
\usepackage{cleveref}

\usepackage{mathtools}
\usepackage[all]{xy}

\numberwithin{equation}{section} 
\newtheorem{thm}[equation]{Theorem} 

\newtheorem{prop}[equation]{Proposition}
\newtheorem{lemma}[equation]{Lemma}

\newtheorem{conjecture}[equation]{Conjecture}

\theoremstyle{definition}
\newtheorem{defin}[equation]{Definition}
\newtheorem{example}[equation]{Example}
\newtheorem{rmk}[equation]{Remark}

\newcommand{\Z}{\mathbb Z}
\newcommand{\Spec}{\operatorname{Spec}}

\newcommand{\Gr}{\operatorname{Gr}}
\newcommand{\Char}{\operatorname{char}}

\newcommand{\ind}{\operatorname{ind}}
\newcommand{\PGL}{\operatorname{PGL}}
\newcommand{\Mat}{\operatorname{M}}

\newcommand{\GL}{\operatorname{GL}}
\newcommand{\UD}{\operatorname{UD}}
\renewcommand{\P}{\mathbb P}

\newcommand{\rd}{\operatorname{rd}}

\newcommand{\Sym}{\mathfrak{S}}

\newcommand{\ed}{\operatorname{ed}}
\newcommand{\Gal}{\operatorname{Gal}}
\newcommand{\per}{\operatorname{per}}

\newcommand{\Mor}{\operatorname{Mor}}
\newcommand{\mc}[1]{\mathcal{#1}}
\newcommand{\cl}{\overline}

\newcommand{\on}[1]{\operatorname{#1}}

\newcommand{\SB}{\mathop{\mathrm{SB}}\nolimits}

\newcommand{\eqdef}{\mathrel{\smash{\overset{\mathrm{\scriptscriptstyle def}} =}}}

\author{Zinovy Reichstein}
\address[Reichstein]{Department of Mathematics\\
	University of British Columbia\\
	Vancouver, BC V6T 1Z2\\Canada}
\email{reichst@math.ubc.ca}
\thanks{Zinovy Reichstein was partially supported by
	National Sciences and Engineering Research Council of
	Canada Discovery grant  RGPIN-2023-03353.}

\author{Federico Scavia}
\address[Scavia]{CNRS\\
	Institut Galil\'ee\\
	Universit\'e Sorbonne Paris Nord\\
	99 avenue Jean-Baptiste Cl\'ement, 93430\\ 
	Villetaneuse, France}
\email{scavia@math.univ-paris13.fr}



\title[Specialization and rigidity]{Specialization and rigidity}
\begin{document}
	
	\begin{abstract} We describe a general method, originated by Ofer Gabber, for showing that a very general fiber in a family has certain properties. We illustrate this method with concrete examples taken from algebraic dynamics, the rationality problem for algebraic varieties, Galois theory, quadratic form theory and the theory of central simple algebras.
	\end{abstract} 
	
	\maketitle

    \tableofcontents
	
    \section{Introduction}
    Let $S$ be a scheme and $\mathcal{P}$ be a property of points of $S$. That is, for each point $s$ of $S$, property $\mathcal{P}$ is either true or false for $s$.    We would like to describe the locus $\Lambda$ of points of $S$ that have property $\mathcal{P}$. There is nothing interesting to say
    about $\Lambda$ in full generality. However, as we shall see, $\Lambda$ acquires a great deal of structure under some relatively mild conditions on $\mathcal{P}$.
 
We say that property $\mathcal{P}$ of points of $S$ satisfies 
the \emph{specialization condition} if whenever $s\in S$ has property $\mathcal{P}$ and $s_0 \in S$ is in the Zariski closure of $s$, then $s_0$ also has property $\mathcal{P}$. Given a morphism of schemes $\pi\colon S\to S'$ and a property $\mc{P}'$ of points of $S'$, we say that $\mc{P}$ \emph{descends} to $\mc{P}'$ along $\pi$ if, for all $s\in S$, the point $s$ satisfies $\mc{P}$ if and only if $\pi(s)$ satisfies $\mc{P}'$. 

  The purpose of this paper is to highlight the following observation, extracted from Gabber~\cite[Appendice]{colliot2002exposant}.
 
\begin{thm} \label{thm.main}
Let $S$ be a scheme and $\mathcal{P}$ be a property of points of $S$. Suppose that there exist a countable scheme $S'$, a morphism $\pi \colon S \to S'$ and a property $\mathcal{P}'$ of points of $S'$ such that $\mathcal{P}$ descends to $\mathcal{P}'$ along $\pi$.
Assume further that $\mc{P}'$ satisfies the specialization condition. Then there exists a countable collection $\{S_j\}_{j\geqslant 1}$ of closed subschemes of $S$ such that a point $s \in S$  
has property $\mathcal{P}$ if and only if $s$ lies in $S_j$ for some $j \geqslant 1$.
\end{thm}

\begin{proof} Let $\Lambda\subset S$ be the locus of points in $S$ which satisfy $\mc{P}$, and let $\Lambda'\subset S'$ be the locus of points in $S'$ which satisfy $\mc{P}'$. As $\mc{P}'$ satisfies the specialization condition, $\Lambda'$ is stable under 
specialization. Hence $\Lambda'$ is a union of closed subsets of $S'$; see~\cite[0EES]{stacks-project}. Since $S'$ is countable, 
this union can be taken to be countable. As $\mc{P}$ descends to $\mc{P}'$ along $\pi$, we have $\Lambda=\pi^{-1}(\Lambda')$. By the continuity of $\pi$, 
we conclude that $\Lambda$ is a countable union of closed subsets of $S$.
\end{proof}

Note that the countable scheme $S'$ and property $\mathcal{P}'$ in \Cref{thm.main} can often be constructed in a natural way. When $S$ is of finite type over a field $k$, $S$ descends to a $k'$-scheme $S'$ of finite type, where $k'$ is some finitely generated subfield of $k'$. Moreover, if property $\mathcal{P}$ is determined by some additional structure (e.g., a scheme, a morphism, etc.) which is of finite type over $S$, then after enlarging $k'$ by adjoining at most finitely many elements, $\mc{P}$ also descends to a property $\mc{P}'$ of points of $S'$.

\smallskip

Although its proof is very short, in this general form \Cref{thm.main} turns out to be quite versatile. It can either be used directly or leveraged in the following setting: 
\begin{equation}\label{eq:setting}
\begin{array}{c}
\text{We are given a morphism $f \colon X \to Y$ of algebraic stacks over $S$, and} \\
\text{property $\mc{P}$ of a point $s \in S$ is a property of the morphism $f_s \colon Y_s \to Z_s$.}
\end{array}
\end{equation}

The following three situations are special cases of \eqref{eq:setting}.
\[
\tag{A} \label{eq:A}
\text{We are given an $S$-scheme $X$, and $\mathcal{P}$ is a property of the fibers $X_s$.}
\]
Here we take $Y$ to be $S$ and $f\colon X \to Y$ to be the structure morphism.
\[
\tag{B} \label{eq:B}
\text{We are in the situation of \eqref{eq:setting}, where $f \colon X \to Y$ is a morphism of schemes.} 
\]
\begin{equation}\tag{C}\label{eq:C}
\begin{array}{c}
\text{We are given an $S$-scheme $X$, a $G$-torsor $\alpha \colon E \to X$ under some $S$-group $G$,} \\
\text{and property $\mathcal{P}$ of point $s \in S$ is a property of the $G_s$-torsor $\alpha_s \colon E_s \to X_s$.}
\end{array}
\end{equation}
Here we take $Y$ to be the classifying stack of $G$ over $S$ and $f\colon X \to Y$ to be  the morphism which classifies $\alpha$.

Gabber~\cite[Appendice]{colliot2002exposant} used the setting of \eqref{eq:A} to show that if $\alpha$ is an \'etale cohomology class on $X$, then $\alpha$ vanishes at the generic points of the geometric fiber $X_s$ if and only if the geometric point $s$ lies in a countable union of closed subschemes of $S$. Maulik and Poonen~\cite[Proposition 3.8]{maulik2012neron} showed that the same is true for the jumping locus of the N\'eron-Severi rank of $X_s$. 

In~\cite{reichstein2022behavior} we studied how essential dimension varies in a family of torsors in the setting of~\eqref{eq:C}; see also \cite{reichstein2023behavior} and Example~\ref{ex.ed}. We noticed that some of the arguments in~\cite{reichstein2022behavior} and \cite{reichstein2023behavior} are not specific to essential dimension, and can be applied to many properties of the geometric fibers of $X \to Y$ over $S$. This paper has resulted from our attempts to find a general framework where such results can be proved in a uniform way.  

The remainder of this paper is structured as follows. In Section~\ref{sect.dynamics} we give an application of Theorem~\ref{thm.main} to algebraic dynamics. Starting from \Cref{sect.morphisms} we move to the general setting of \eqref{eq:setting}. More specifically, we are given a morphism $f \colon X \to Y$ of algebraic stacks over $S$, and property $\mathcal{P}$ of point $s \in S$ is actually a property of the morphism $f_{\cl{s}} \colon X_{\cl{s}} \to Y_{\cl{s}}$, where $\cl{s}$ is a geometric point with image $s$. In order for this to be well defined, property $\mc{P}$ must satisfy the \emph{rigidity condition}, that is, be independent of the choice of geometric point $\cl{s}$ lying over $s$. 

In \Cref{sect.rationality}, we show how celebrated theorems on the behavior of rationality properties in families of smooth projective varieties (rationality, stable rationality, universal $CH_0$-triviality) fit into the setting \eqref{eq:A}. 

In \Cref{sect.rat-sections}, due to Angelo Vistoli, we assume that $f \colon X \to Y$ is a projective morphism of schemes in setting \eqref{eq:B} and consider the locus of points $s \in S$ such that $f_s$ admits a geometric rational section. When $Y$ is flat over $S$, this locus once again turns out to be a countable union of closed subschemes of $S$. We give several applications of this result in Sections~\ref{sect.examples},
\ref{sect.reduction-to-parabolic} and~\ref{sect.hilbert-irreducibility}.

Starting from~\Cref{sect.properties-of-torsors} we work in the setting \eqref{eq:C}. This motivates considering algebraic stacks, and not just schemes, in \Cref{sect.morphisms}. 

In~\Cref{sect.reduction-to-general} we consider the property that the $G$-torsor $E_{\cl{s}} \to X_{\cl{s}}$ admits reduction of structure to $H$ for a fixed group homomorphism $H \to G$. Once again, we show that this happens if and only if $\cl{s}$ lies in a countable union of closed subschemes of
$S$. (The case, where $H$ is a parabolic subgroup of $G$ is considered earlier in~\Cref{sect.reduction-to-parabolic}.) As an application, in \Cref{sect.symbol-length} we study how the symbol length of a Brauer class varies in a family of central simple algebras. 

In~\Cref{sect.strongly-unramified} we construct a strongly unramified $G$-torsor with specified properties; see \Cref{prop.combination} for a precise statement. We then give two applications of this construction with $G = \PGL_n$: in \Cref{sect.unramified-non-crossed} we show that there exists an unramified non-crossed product division algebra over a $6$-dimensional field, and in~\Cref{sect.period-index} we give a new proof of a
 result of Starr and de Jong about unramified counterexamples to the period-index problem.

The overall structure of the paper is shown in the flowchart below.
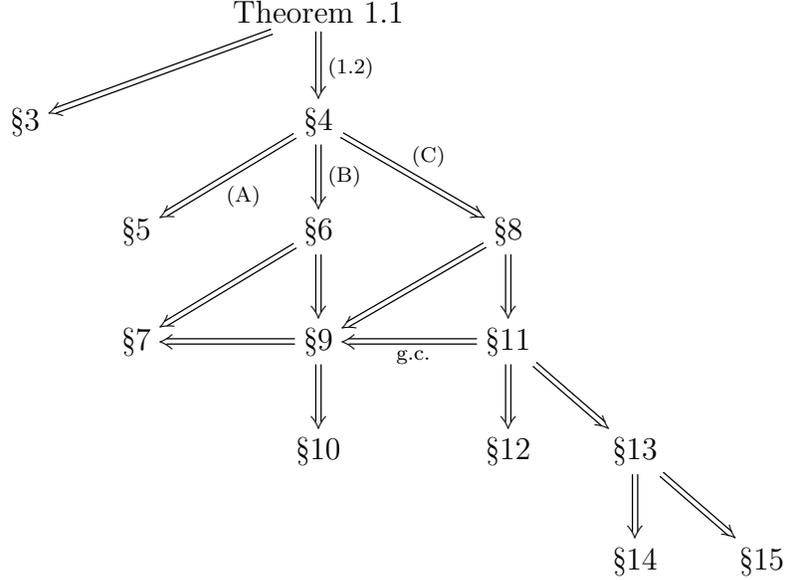
\[  \xymatrix{   &    & 
\text{Theorem~\ref{thm.main}} \ar@{=>}[dll] \ar@{=>}[d]^{\eqref{eq:setting}} & & & \\
\text{\S 3} & & \text{\S 4}
\ar@{=>}[dl]^{\eqref{eq:A}} \ar@{=>}[d] ^{\eqref{eq:B}}\ar@{=>}[dr]^{\eqref{eq:C}} 
& &  & \\ 
 & \text{\S 5} 
& \text{\S 6} \ar@{=>}[d] \ar@{=>}[dl]
& \text{\S 8} \ar@{=>}[dl] \ar@{=>}[d]  &  &  \\
 &\text{\S 7}  & \text{\S 9} \ar@{=>}[l] \ar@{=>}[d]  & \text{\S 11} \ar@{=>}[l]^{\text{g.c.}} \ar@{=>}[d] \ar@{=>}[dr] &  & \\
 & & \text{\S 10} &   \text{\S 12} & \text{\S 13} \ar@{=>}[d] \ar@{=>}[dr] & \\
 &  &  & & \text{\S 14} & \text{\S 15} } \]

Here an arrow $\S M \Longrightarrow \S N$ indicates that the material of Section M is used in Section N.

The labels \eqref{eq:setting}, \eqref{eq:A},~\eqref{eq:B}, and~\eqref{eq:C} refer to the settings, where Theorem~\ref{thm.main} is applied; see~above. 

The label ``g.c." stands for ``good characteristic". This arrow indicates that \Cref{prop.parabolic} follows from
Theorem~\ref{thm.reduction-of-structure} if the parabolic subgroup $P$ is in good characteristic but not otherwise.

\section{Notational conventions}
Let $k$ be a field, and let $X$ be a $k$-scheme.
\begin{itemize}
    \item A \emph{point} $x\in X$ is a point of the underlying topological space of $X$. For every point $x\in X$, we write $k(x)$ for the residue field of $x$, which is a field extension of $k$.   Given two points $x,y\in X$, we say that $y$ is a \emph{specialization} of $x$ if $y$ belongs to the Zariski closure of $\{x\}$ in $X$.

    \item For a field $K$, a \emph{$K$-point} $x\in X(K)$ is a morphism $\on{Spec}(K)\to X$.
    \item A \emph{geometric point} $x$ of $X$ is a $K$-point of $X$ for some algebraically closed field $K$. If $x\colon \on{Spec}(K)\to X$ is a geometric point, we sometimes also write $k(x)$ for $K$. If $P\in X$ is the image of the geometric point $x$, we say that $x$ \emph{lies over $P$}. If $Z\subset X$ is a a closed subscheme, we say that the geometric point $x$ \emph{lies in $Z$} if the image of $x$ is contained in $Z$. 
\end{itemize}

Let $V$ be a $k$-variety, that is, a geometrically reduced separated $k$-scheme of finite type. A generic point of $V$ is a point of $V$ of codimension $0$, or equivalently a point of $V$ whose closure is an irreducible component of $V$. We let $k(V)$ be the product of the function fields of the irreducible components of $V$, that is, $k(V)\coloneqq \prod_{i=1}^rk(\eta_i)$, where $\eta_1,\dots,\eta_r\in V$ are the generic points of $V$. For every field extension $k'/k$ and every $x\in V(k')$, we let $V_{k'}\coloneqq V\times_kk'$ and we let $x_{k'}\in V(k')$ be the $k'$-point of $V(k')$ induced by $x$.

An algebraic $k$-stack is an algebraic stack for the fppf topology over $k$, in the sense of \cite[Tag 026O]{stacks-project}. If $G$ is a $k$-group, that is, a group scheme of finite type over $k$, acting on a $k$-scheme $X$, we write $[X/G]$ for the quotient stack, which is an algebraic $k$-stack; see \cite[Tags 04UI, 0CQJ]{stacks-project}. In particular, the classifying stack $B_kG\coloneqq [\on{Spec}(k)/G]$ is an algebraic $k$-stack. Giving a morphism $Y\to B_kG$ is equivalent to giving an fppf $G$-torsor $X\to Y$. Recall that if the $k$-group $G$ is smooth (equivalently, geometrically reduced), every fppf $G$-torsor is \'etale-locally trivial; see \cite[Corollaire 17.16.3(ii)]{ega4}.

\section{An example from algebraic dynamics}
\label{sect.dynamics}

Let $k$ be a field, let $X$ be a $k$-variety, and let $f\colon X\to X$ be a morphism. For every field extension $K/k$ and every $x\in X(K)$ we define $V_x \subset X_K$ as the Zariski closure of $\{f^n(x)\}_{n\geqslant 0}$ in $X_K$. Thus $V_x$ is the smallest $f$-invariant closed subscheme of $X_K$ containing the $K$-point $x$.

\begin{prop} \label{prop.dynamics}
     Let $k$ be a field, let $X$ be a $k$-variety, let $f\colon X\to X$ be a morphism, and let $d\geqslant 0$ be an integer. There exists a countable collection $\{X_j\}_{j\geqslant 1}$ of closed subschemes $X_j\subset X$ such that for every field $K$ containing $k$ and every $x \in X(K)$, we have $\dim_{K}(V_x) \leqslant d$ if and only if $x \in X_j(K)$ for some $j \geqslant 1$.
\end{prop}

\begin{example}
Let $k$ be a field, let $E$ be an elliptic curve over $k$, let $X\coloneqq E\times_k E$, and consider the morphism $f \colon X\to X$ given by $f(a, b) = (a, a + b)$. For every field extension $K/k$ and every $K$-point $(a,b)$ of $X$, the subvariety $V_{(a, b)}\subset X_K$ is a finite set of points if $a$ is torsion, and is equal to $\{ a \} \times_K E_K$ if $a$ is not torsion. For every $n\geqslant 1$, let $X_n \coloneqq E[n] \times_k E$. The collection $\{X_n\}_{n\geqslant 1}$ satisfies the conclusion of  \Cref{prop.dynamics} for $d=0$.
\end{example}

\Cref{prop.dynamics} is not new; see~\cite[Theorem 4.1]{amerik-campana}, where $X$ is assumed to be a compact K\"ahler manifold and $f$ a dominant meromorphic map $X \dasharrow X$, or~\cite[Theorem 1.2]{bell-ghioca-reichstein}, where $X$ is assumed to be quasi-projective and $f \colon X \dasharrow X$ is a dominant rational map. Our purpose here to show that \Cref{prop.dynamics} readily follows from \Cref{thm.main}. Our proof will rely on the following two lemmas.

\begin{lemma} \label{lem.orbit-closure-rigidity}
    Let $k$ be a field, let $X$ be a $k$-variety, let $f\colon X\to X$ be a morphism, let $K/k$ be a field extension, and let $x\in X(K)$. For every field extension $L/K$, we have $((V_x)_L)_{\on{red}}=V_{x_L}$ as closed subschemes of $X_L$. In particular, the dimension of $V_x$ only depends on the image of $x$ in $X$. 
\end{lemma}

\begin{proof}
    We recall that for every $K$-scheme $Z$ and every subset $\Sigma\subset Z(K)$ that is Zariski-dense in $Z_{\on{red}}$, the image of $\Sigma$ in $Z(L)$ is Zariski-dense in $(Z_L)_{\on{red}}$. Indeed, since Zariski-density may be checked Zariski locally, we may assume that $Z=\on{Spec}(A)$ for some $K$-algebra $A$. Since $\Sigma$ is Zariski-dense in $Z$, the kernel $I$ of the $k$-algebra homomorphism $A\to \prod_{s\in \Sigma} K(s)$ is a nilpotent ideal. Because $L$ is flat over $k$, the kernel of the base change map $A\otimes_KL\to\prod_{s\in \Sigma} L(s_L)$ is equal to $I\otimes_KL$, and hence is nilpotent, that is, $\Sigma_L$ is Zariski dense in $(Z_L)_{\on{red}}$.
    
    We now come to the proof of \Cref{lem.orbit-closure-rigidity}. Replacing $X$ by $X_K$ and $f$ by $f_K$, we may assume that $K=k$. Since $((V_x)_L)_{\on{red}}\subset X_L$ is a closed $f$-invariant subscheme containing $x_L$, we have the inclusion $V_{x_L}\subset ((V_x)_L)_{\on{red}}$. We may thus replace $X$ by $V_x$, that is, we may assume that $\{f^n(x)\}_{n\geqslant 0}$ is Zariski-dense in $X$, and we must show that $\{f^n(x_L)\}_{n\geqslant 0}$ in $X(L)$ is Zariski-dense in $(X_L)_{\on{red}}$. This follows from the first part of the proof.

    The final statement follows from the rest and the observation that every $K$-point $x\colon \on{Spec}(K)\to X$ with image $P\in X$ factors as $\on{Spec}(K)\to\on{Spec}(k(P))\to X$.
\end{proof}

A point $x \in X$ determines a $k(x)$-point $x \colon \Spec(k(x)) \to X$. Hence, for every $x\in X$, we may consider $V_x\subset X_{k(x)}$.

\begin{lemma} 
\label{lem.orbit-closure-specialization}
    Let $k$ be a field, let $X$ be a $k$-variety, and let $f\colon X\to X$ be a morphism. Let $x, y \in X$ be such that $y$ is a specialization of $x$. Then $\dim(V_x)\geqslant \dim(V_y)$.
\end{lemma}
    
\begin{proof} By~\cite[Proposition 7.1.9]{ega2},
    there exist a discrete valuation ring $R$ containing $k$, with fraction field $K$ and residue field $\kappa$, and a $k$-scheme morphism $\on{Spec}(R)\to X$ sending the generic point $\eta\in \on{Spec}(R)$ to $x$ and the closed point $s\in \on{Spec}(R)$ to $y$. 
    Letting $X_R\coloneqq X\times_kR$, this morphism induces a section $\sigma\colon \on{Spec}(R)\to X_R$ of the projection $X_R\to \on{Spec}(R)$ such that $\sigma(\eta)=x_K$ and $\sigma(s)=y_\kappa$. 
    
    We let $\Sigma\subset X_R$ be the image of the closed immersion $\sigma$, that is, the reduced closed subscheme of $X_R$ with underlying set $\{x_K,y_\kappa\}$, and we let $V_\sigma\subset X_R$ be the Zariski closure of $\{f_R^n(\Sigma)\}_{n\geqslant 0}$ in $X_R$, that is, the smallest closed subscheme of $X_R$ which contains $f_R^n(\Sigma)$ for all $n\geqslant 0$. Since $y_\kappa$ is a specialization of $x_K$, for all $n\geqslant 0$ the point $f^n(y_\kappa)$ is a specialization of $f^n(x_K)$, and hence $V_\sigma$ is also the Zariski closure of $\{f_K^n(x_K)\}_{n\geqslant 0}$, that is, $V_\sigma$ is the Zariski closure of $V_\sigma$. Thus, by \cite[Tag 02CS]{stacks-project}, $V_\sigma$ is a proper model of $V_{x_K}$ in the sense of \cite[Tag 02CR]{stacks-project}, i.e. $V_\sigma$ is $R$-flat and $(V_\sigma)_K=V_{x_K}$. Since $y_\kappa$ is contained in $\Sigma$, we have $V_{y_\kappa}\subset (V_\sigma)_\kappa$. The $R$-flatness of $V_\sigma$ implies that  $\dim((V_\sigma)_K)\geqslant \dim((V_\sigma)_\kappa)$; see \cite[Lemma 0D4H]{stacks-project}. All in all, we obtain
    \[\dim(V_x)=\dim(V_{x_K})=\dim((V_\sigma)_K) =\dim((V_\sigma)_\kappa)\geqslant\dim(V_{y_\kappa})=\dim(V_y),\]
    where the first and last equalities come from \Cref{lem.orbit-closure-rigidity}.  
\end{proof}

\begin{proof}[Proof of \Cref{prop.dynamics}] Our goal is to apply \Cref{thm.main}, with $\mathcal{P}$ being the following property of a point $x \in X$: $\dim (V_x) \leqslant d$.

    Since $X$ is of finite type over $k$, both $X$ and $f$ descend to some finitely generated subfield $k' \subset k$.
    In other words $X=(X')_k$ and $f=(f')_k$ for some $k'$-variety $X'$ and some morphism of $k'$-schemes $f'\colon X'\to X'$. Let $\pi\colon X\to X'$ be the projection map. By \Cref{lem.orbit-closure-rigidity}, for every scheme-theoretic point $x \in X$, we have $\dim (V_x)=\dim(V_{\pi(x)})$. Hence, property $\mathcal{P}$ descends to property 
    $\mathcal{P}'$ along $\pi$, where by definition a point $x' \in X'$ has property $\mathcal{P}'$ if and only if
    $\dim  (V_{x'}) \leqslant d$. Note that since $k'$ is finitely generated over the prime field and $X'$ is of finite type over $k'$, the underlying set of $X'$ is countable.

   In order to apply \Cref{thm.main}, it only remains to establish the specialization condition for property $\mathcal{P}'$. To simplify the notation, we may replace $k$ by $k'$ and $X$ by $X'$, i.e., we may assume that $k$ is countable. Now we only need to establish the specialization condition for $\mathcal{P}$. This is done in~\Cref{lem.orbit-closure-specialization}. \Cref{prop.dynamics} now follows from \Cref{thm.main}.
\end{proof}

\begin{rmk} The set of morphisms $X \to X$ forms a monoid $\Mor_k(X)$ under composition. The set of powers of $f$ forms a submonoid $\langle f \rangle = \{ {\rm id}, f, f^2, f^3, \ldots \}$ of $\Mor_k(X)$. \Cref{prop.dynamics} remains valid if we replace $\langle f \rangle$ by a countable submonoid of $\Mor_k(X)$. In this setting we define $V_x$ to be the closure of $\{ m(x) \, | \, m \in M \}$. The proof goes through unchanged.
\end{rmk} 

\section{Properties of geometric morphisms} 
\label{sect.morphisms}

\begin{defin}\label{defin:geometric-morphism}
    Let $k$ be a field. A \emph{property of geometric morphisms over $k$} is a property of pairs $(k',\varphi)$, where $k'$ is an algebraically closed field containing $k$ and $\varphi\colon V\to W$ is a morphism of algebraic $k'$-stacks.
\end{defin} 

\begin{defin}\label{defin:rigidity}
    Let $k$ be a field, and let $\mc{P}$ be a property of geometric morphisms over $k$. We say that $\mc{P}$ satisfies the \emph{rigidity condition} if, for every extension of algebraically closed fields $k''/k'$ containing $k$ and for every pair $(k',\varphi)$ as in \Cref{defin:geometric-morphism}, property $\mc{P}$ holds for $(k',\varphi)$ if and only if it holds for the base change $(k'',\varphi_{k''})$.
\end{defin}

Let $k$ be a field, and let $\mc{P}$ be a property of geometric morphisms over $k$ which satisfies the rigidity condition. For every $k$-scheme $S$ and every morphism of algebraic $S$-stacks $f\colon X\to Y$, we may define a property $\mc{P}_{S,f}$ of points of $S$ as follows: For every $s \in S$, the point $s$ satisfies $\mc{P}_{S,f}$ if and only if the pair $(k(\cl{s}),f_{\cl{s}}\colon X_{\cl{s}}\to Y_{\cl{s}})$ satisfies $\mc{P}$ for some geometric point $\cl{s}$ of $S$ lying over $s$ (equivalently, for every
geometric point $\cl{s}$ of $S$ lying over $s$, because of the rigidity condition).

\begin{defin}\label{defin:specialization}
    Let $k$ be a field, let $\mc{P}$ be a property of geometric morphisms over $k$ which satisfies the rigidity condition, let $S$ be a $k$-scheme, and let $f\colon X\to Y$ be a morphism of algebraic $S$-stacks. We say that $\mc{P}$ satisfies the \emph{specialization condition} with respect to $(S,f)$ if property $\mc{P}_{S,f}$ of points of $S$ satisfies the specialization condition (in the sense of \Cref{thm.main}).
\end{defin}

\Cref{thm.main} has the following consequence for properties of geometric morphisms.

\begin{thm}\label{thm.main1}
    Let $k$ be a field, let $k_0$ be the prime field of $k$, let $\mc{P}$ be a property of geometric morphisms over $k_0$, let $S$ be a $k$-scheme of finite type, and let $f\colon X\to Y$ be a morphism of algebraic $S$-stacks of finite type.  Suppose that $\mc{P}$ satisfies the rigidity condition, as well as the specialization condition with respect to $(S,f)$. There exists a countable family $\{S_j\}_{j\geqslant 1}$ of closed subschemes of $S$ such that, for every geometric point
    $s$ of $S$, the following are equivalent:
    
    \smallskip (i) the fiber $f_s\colon X_s \to Y_s$ over $s$ has property $\mathcal{P}$, and 
    
    \smallskip (ii) $s$ lies in $S_j$ for some $j \geqslant 1$.    
\end{thm}

\begin{proof}
We must show that the locus of points $s \in S$ satisfying $\mc{P}_{S,f}$ is a countable union of closed subschemes of $S$. There exist a finitely generated subfield $k'\subset k$, a $k'$-scheme $S'$ such that $S=S'\times_{k'}k$ and a morphism $f'\colon X'\to Y'$ of algebraic $k'$-stacks such that $f=(f')_k$. We let $\pi\colon S\to S'$ be the projection map, which is a faithfully flat morphism. 

For every $s'\in S'$ and every geometric point $\cl{s}$ of $S$ such that the geometric point $\pi\circ \cl{s}$ of $S'$ lies over $s'$, the morphism $f_{\cl{s}}\colon X_{\cl{s}}\to Y_{\cl{s}}$ is a base change of $f_{s'}\colon X_{s'}\to Y_{s'}$ along the inclusion $k(s')\subset k(\cl{s})$. Since $\mc{P}$ satisfies the rigidity condition, this implies that for all $s_1,s_2\in S$ such that $\pi(s_1)=\pi(s_2)$, the point $s_1$ satisfies $\mc{P}_{S,f}$ if and only if so does $s_2$. In other words, property $\mc{P}_{S,f}$ descends along $\pi$ to a property $\mc{P}'$ of points of $S'$: namely, a point $s'\in S'$ satisfies $\mc{P}'$ if and only if for some (equivalently, every) $s\in S$ such that $\pi(s)=s'$, the point $s$ satisfies $\mc{P}_{S,f}$.

As the scheme $S'$ is countable, by \Cref{thm.main} it only remains to prove that $\mc{P}'$ satisfies the specialization condition. Let $s_0',s_1'\in S'$ be such that $s_1'$ is in the closure of $s_0'$, and suppose that $s_0'$ satisfies $\mc{P}'$. Since $\pi$ is surjective, there exists $s_1\in S$ such that $\pi(s_1)=s_1'$. As $\pi$ is flat, by \cite[Tag 03HV]{stacks-project} there exists $s_0\in S$ such that $\pi(s_0)=s_0'$. Since the point $s_0'\in S'$ satisfies $\mc{P}'$, the point $s_0\in S$ satisfies $\mc{P}_{S,f}$. As $\mc{P}_{S,f}$ satisfies the specialization condition, this implies that $s_1\in S$ also satisfies $\mc{P}_{S,f}$, and hence $s_1'$ satisfies $\mc{P}'$. This shows that $\mc{P}'$ satisfies the specialization condition and completes the proof.
\end{proof}

\begin{rmk}
    Conversely, if the equivalence of (i) and (ii) of \Cref{thm.main1} holds, then property $\mc{P}$ satisfies the rigidity condition with respect to fibers of $f$ above geometric points of $S$ and it satisfies the specialization condition with respect to $(S,f)$.
\end{rmk}

In practice it often suffices to know that the locus of geometric points $s$ of $S$ where property $\mathcal{P}$ holds is {\em contained} in a countable union of closed subschemes of $S$. The following variant of \Cref{thm.main1} provides a somewhat shorter path towards proving such statements.

\begin{defin}\label{defin:very-general}
    Let $k$ be a field, let $S$ be an irreducible $k$-scheme of finite type, and let $\Omega\subset S(k)$ be a subset. We say that $\Omega$ is \emph{very general} if there exist closed subschemes $\{S_j\}_{j\geqslant 1}$ of $S$ such that $S_j\neq S$ for all $j\geqslant 1$ and $S(k)\setminus \Omega$ is contained in $\cup_{j\geqslant 1}S_j(k)$.
\end{defin}

    If $k$ is uncountable, then $\Omega$ is very general $\Longrightarrow$ $\Omega$ is Zariski dense in $S$ $\Longrightarrow$ $\Omega \neq \emptyset$.

\begin{thm}\label{thm.main2} 
    Let $k$ be an algebraically closed field, let $\mc{P}$ be a property of geometric morphisms over $k$, let $S$ be an irreducible $k$-scheme of finite type, let $f\colon X\to Y$ be a morphism of algebraic $S$-stacks of finite type,  and let $s_0\in S(k)$ be such that $f_{s_0}$ does not satisfy $\mc{P}$. Assume that there exist:
    \begin{itemize}
        \item a Zariski dense open subscheme $U\subset S$ such that $\mc{P}$ satisfies the specialization condition with respect to $(U,f_U)$, and
        \item a morphism $\varphi\colon \on{Spec}(R)\to S$, where $R$ is a discrete valuation ring, sending the closed point of $\on{Spec}(R)$ to $s_0$ and the generic point of $\on{Spec}(R)$ to a point of $U$, such that $\mc{P}$ satisfies the specialization condition with respect to $(\on{Spec}(R),f_R)$.
    \end{itemize}
    Then the set of $s\in S(k)$ such that $f_s$ does not satisfy $\mc{P}$ is very general.
\end{thm}

\begin{proof}
    Let $\eta\in U$ be the image of the generic point of $\on{Spec}(R)$ under $\varphi$, and let $\cl{\eta}$ be a geometric point of $U$ lying over $\eta$. We claim that $f_{\cl{\eta}}$ does not satisfy $\mc{P}$. Indeed, otherwise the specialization condition 
 tells us that the geometric morphism $f_{s_0}$ also satisfies $\mc{P}$, contradicting our assumption. This contradiction proves the claim.
    
    Let $\Omega\subset S(k)$ be the set of $s\in S(k)$ such that $f_s$ does not satisfy $\mc{P}$. By \Cref{thm.main1} applied to $f_U$, there exists a countable collection of closed subsets $\{Z_j\}_{j\geqslant 1}$ of $U$ such that every geometric point $u$ of $U$ such that $f_u$ satisfies $\mc{P}$ if and only if $u$ lies in $Z_j$ for some $j\geqslant 1$. In particular, for every $j\geqslant 1$, the geometric point $\cl{\eta}$ does not lie in $Z_j$, and hence $Z_j\neq U$. We conclude that the set $U(k)\setminus \Omega$ is very general in $U$. Since $U$ is open and dense in $S$, the set $S(k)\setminus \Omega$ is also very general.
\end{proof}

We conclude this section with a valuative criterion for the specialization condition.

\begin{lemma}\label{reduce-to-dvr}
    Let $k$ be a field, let $\mc{P}$ be a property of geometric morphisms over $k$ which satisfies the rigidity condition, let $S$ be a locally noetherian $k$-scheme, and let $f\colon X\to Y$ be a morphism of algebraic $S$-stacks. The following statements are equivalent.
    \begin{itemize}
        \item[(1)] Property $\mc{P}$ satisfies the specialization condition with respect to $(S,f)$.
        \item[(2)] For every morphism $\on{Spec}(R)\to S$, where $R$ is a complete discrete valuation ring $R$, property $\mc{P}$ satisfies the specialization condition with respect to $(\on{Spec}(R),f_R)$.
    \end{itemize}
\end{lemma}

\begin{proof} Let $\on{Spec}(R)\to S$ be a morphism,  where $R$ is a complete discrete valuation ring, and let $f_R\colon X_R\to Y_R$ be the base change of $f$ along $\varphi$. By the rigidity condition for $\mc{P}$, the generic (resp. closed) point of $\on{Spec}(R)$ satisfies  property $\mc{P}_{\on{Spec}(R),f_R}$ if and only if its image in $S$ satisfies property $\mc{P}_{S,f}$ for its image in $S$. This shows that (1) $\Longrightarrow$ (2). Moreover, it reduces the proof of the implication (2) $\Longrightarrow$ (1) to the following: For any two points $s_0,s_1\in S$ such that $s_1$ is a specialization of $s_0$, there exist (i) a complete discrete valuation ring $R$ and (ii) a morphism $\on{Spec}(R)\to S$ which sends the generic (resp. closed) point of $\on{Spec}(R)$ to $s_0$ (resp. $s_1$). Since $S$ is locally noetherian, this is true: by \cite[Proposition 7.1.9]{ega2}, one may find there exist a (not necessarily complete) discrete valuation ring $R$ and a morphism $\on{Spec}(R)\to S$ satisfying (ii). Now replace $R$ by its completion.
\end{proof}

	\section{Rationality properties}
	\label{sect.rationality}
	
	The specialization method of Voisin \cite{voisin2015unirational}, subsequently developed by Colliot-Th\'el\`ene and Pirutka \cite{colliot2016hypersurfaces}, allows one to prove that a very general member of a family of projective varieties is not universally $CH_0$-trivial (in particular, not stably rational; see \Cref{rationality-definitions}) given the existence of a suitable mildly singular fiber in the family. In particular, for a smooth projective family  $X\to S$, the locus of universally $CH_0$-trivial fibers is a union of countably many closed subschemes of $S$. Nicaise--Shinder \cite{nicaise2019motivic} and Kontsevich--Tschinkel \cite{kontsevich2019specialization} proved that the same holds for stable rationality and rationality of fibers in smooth projective families, respectively. 
    
    In this section, we explain how \Cref{thm.main} can be used to simplify parts of the proofs of the above theorems. The novelty here is purely expositional: our main point is that \Cref{thm.main} allows one to eliminate the use of Hilbert schemes and Chow varieties. More precisely, let $X\to S$ be a smooth projective morphism of varieties over a field $k$ of characteristic zero. The goal is to prove that the images of the geometric points $s$ of $S$ such that $X_s$ is $k(s)$-rational (or $k(s)$-stably rational, or universally $CH_0$-trivial) lie in a union of countably many closed subschemes of $S$. To achieve this, the first step is to prove that these rationality properties satisfy the specialization condition. The second step consists in combining the specialization condition with a Hilbert scheme or Chow variety argument to obtain the theorem. Our point is that the second step, which can get quite technical, can be simplified: in view of the specialization condition established in the first step, it suffices to prove the rigidity condition, which is a straightforward task, and then appeal to~\Cref{thm.main}. The same idea applies to the \enquote{specialization method} of Voisin and Colliot-Th\'el\`ene--Pirutka. Here we use the more subtle specialization condition proved by Colliot-Th\'el\`ene and Pirutka; see~\Cref{ch0-trivial-specialization}. 
    
	\subsection{Definitions and rigidity condition}
	
	\begin{defin}\label{rationality-definitions}
		Let $k$ be a field, and let $V$ be a $d$-dimensional geometrically integral $k$-variety. We say that $V$ is
		\begin{enumerate}[label=(\roman*)]
			\item $k$-\emph{rational} if there exist an integer $n\geqslant 0$, Zariski dense open subschemes $U\subset V$ and $U'\subset \P^n_k$, and a $k$-isomorphism $U\cong U'$;
			\item $k$-\emph{stably rational} if there exists an integer $m\geqslant 0$ such that $V\times_k \P^m_k$ is $k$-rational;
			\item\label{ch0-trivial-def} (for proper $V$) \emph{universally $CH_0$-trivial} if for every field extension $F/k$, the degree map $CH_0(V_F)\to \Z$ is an isomorphism.
		\end{enumerate}        
        We have (i) $\Rightarrow$ (ii) $\Rightarrow$ (iii). 
	\end{defin}

	\begin{rmk}\label{rmk-decomposition-diagonal}
		Let $k$ be a field, and let $V$ be a smooth proper geometrically integral $k$-variety. We say that $V$ \emph{admits a Chow decomposition of the diagonal} if there exist a divisor $D\subset V$, a $d$-dimensional cycle $Z\in Z_d(V\times_k V)$ supported on $D\times_k V$, and a zero-cycle $\alpha$ of degree $1$ such that $\Delta_V = Z+V\times_k \alpha$ in $CH_d(V\times_k V)$. 
		
        By \cite[Proposition 1.4]{colliot2016hypersurfaces}, the following are equivalent:
        \begin{enumerate}
            \item $V$ is universally $CH_0$-trivial;
            \item $V$ admits a zero-cycle of degree $1$ and the degree map $CH_0(V_{k(V)})\to\Z$ is an isomorphism;
            \item $V$ admits a Chow decomposition of the diagonal.
        \end{enumerate}
	\end{rmk} 
	
	\begin{prop}\label{rationality-properties-rigidity}
		Let $k$ be a field, and let $\mc{P}$ be one of the properties of geometric morphisms over $k$ in \Cref{rationality-definitions}. Then $\mc{P}$ satisfies the rigidity condition.
	\end{prop}

    More precisely, in case (i) (resp. (ii), resp. (iii)) of \Cref{rationality-definitions}, property $\mc{P}$ is defined by: a pair $(k', \varphi\colon V\to W)$ as in \Cref{defin:geometric-morphism} satisfies $\mc{P}$ if and only if $W=\on{Spec}(k')$ and $V$ is $k'$-rational (resp. $k'$-stably rational, resp. universally $CH_0$-trivial).
    
	\begin{proof}
It suffices to prove the following: If $K/k$ is an extension of algebraically closed fields and $V$ a smooth projective geometrically integral $k$-variety such that $V_K$ satisfies $\mc{P}$, then so does $V$. 

Suppose that $\mc{P}$ is the rationality property, as in part (i) of
\Cref{rationality-definitions}. A spreading-out argument yields a smooth $k$-algebra of finite type $k\subset A\subset K$, two
$A$-fiberwise dense open subschemes $U\subset V_A$ and $U'\subset \P^d_A$, and an isomorphism of $A$-schemes $\varphi\colon U\xrightarrow{\sim}U'$. Here $d = \dim(V)$, as in \Cref{rationality-definitions}.

		Because $k$ is algebraically closed and $A$ is of finite type over $k$, there exists a $k$-point $x\colon \on{Spec}(k)\to \on{Spec}(A)$. Pulling back along $x$ shows that $V$ satisfies $\mc{P}$.

The same argument applies in the case, where $\mc{P}$ is the property of being stably rational, as in \Cref{rationality-definitions}(ii): we just need to replace $V$ by $V \times \mathbb{P}^m$ for a suitable $m \geqslant 0$.
		
		Finally, assume that $\mc{P}$ is the property of being universally $CH_0$-trivial, as in  \Cref{rationality-definitions}(iii). We will use the equivalent formulation of $\mc{P}$ involving the Chow decomposition of the diagonal, as recalled in~\Cref{rmk-decomposition-diagonal}. Since Chow groups commute with filtered colimits of field extensions, there exists a finitely generated field subextension $k\subset F\subset K$ such that $V_F$ admits a decomposition of the diagonal. Let $W$ be a smooth integral $k$-variety with function field $F$. By a limit argument, shrinking $W$ if necessary, we may assume that the second projection $V\times_kW\to W$ admits a generic decomposition of the diagonal, in the sense of Voisin \cite[(3.52)]{voisin2019birational}: There exist
        \begin{itemize}
            \item a rational section of the second projection $V\times_kW\to W$ with image $Z\subset V\times_kW$,
            \item a proper closed subscheme $C\subset V\times_kW$, and
            \item a cycle $T \subset V\times_kV\times_kW \to V\times_kW$ which is supported over $C$,
        \end{itemize} 
        such that $V\times\Delta = V\times_kZ + T$ in $CH(V\times_kV\times_kW)$. Shrinking $W$ if necessary, we may assume that this rational section is a morphism. Let $x\colon \on{Spec}(k)\to W$ be a $k$-point. Because $W$ is smooth, pullback along $x$ is defined and yields the required decomposition of the diagonal for $V$. 
	\end{proof}
	
	\subsection{Rationality properties in smooth projective families}

	\begin{thm}[Nicaise--Shinder, Kontsevich--Tschinkel, Voisin, Colliot-Th\'el\`ene--Pirutka]\label{countable-union-rationality-properties}
Let $k$ be a field of characteristic zero, let $S$ be a $k$-scheme of finite type, and let $f\colon X\to S$ be a smooth projective morphism with geometrically integral fibers. For every geometric point $s\in S$, consider the following assertions:
\begin{enumerate}
    \item[$(i)$] the $k(s)$-variety $X_s$ is rational;
    \item[$(ii)$] the $k(s)$-variety $X_s$ is stably rational;
    \item[$(iii)$] the $k(s)$-variety $X_s$ is universally $CH_0$-trivial.
\end{enumerate}
Let $(\ast)$ be one of $(i)$, $(ii)$, or $(iii)$. There exists a countable family of closed subschemes $\{S_j\}_{j\geqslant 1}$ of $S$ such that, for every geometric point $s$ of $S$, the $k(s)$-variety $X_s$ satisfies $(*)$ if and only if $s$ lies in one of the $S_j$.
	\end{thm}

	\begin{proof}
	Let $\mc{P}$ be the property of geometric morphisms corresponding to $(*)$, as in the statement of \Cref{rationality-properties-rigidity}. By \Cref{rationality-properties-rigidity}, property $\mc{P}$ satisfies the rigidity condition. We now show that $\mc{P}$ satisfies the specialization condition with respect to $(S,f)$. By \Cref{reduce-to-dvr}, we may assume that $S=\on{Spec}(R)$, where $R$ is a complete discrete valuation ring containing $k$ and with residue field $\kappa$. Since $\kappa$ contains $k$, it is of characteristic zero, and hence by Cohen's structure theorem \cite[Lemma 0C0S]{stacks-project} we have a $\kappa$-algebra isomorphism $R\cong \kappa[\![t]\!]$. We may thus assume that $R=\kappa[\![t]\!]$. We may also suppose that $\kappa$ is algebraically closed. Since $\kappa$ is of characteristic zero, an algebraic closure of $\kappa(\!(t)\!)$ is given by the field of Puiseux series $\cl{\kappa(\!(t)\!)}=\cup_{n\geqslant 1}\kappa(\!(t^{1/n})\!)$.
        
       Suppose that $\mc{P}$ is property (i). The geometric fiber $X_{\cl{\kappa(\!(t)\!)}}$ is rational and hence the $\kappa(\!(t^{1/n})\!)$-variety $X_{\kappa(\!(t^{1/n})\!)}$ is rational for some $n\geqslant 1$. Replacing $t$ by $t^{1/n}$, we may assume that $n=1$, that is, the $\kappa(\!(t)\!)$-variety $X_{\kappa(\!(t)\!)}$ is rational. By Kontsevich--Tschinkel \cite{kontsevich2019specialization}, this implies that $X_{\kappa}$ is $\kappa$-rational, that is, $s_0$ satisfies $\mc{P}$. Thus property (i) satisfies the specialization condition.

       The proof of the specialization condition for property (ii) is entirely analogous, replacing the result of Kontsevich--Tschinkel by that of Nicaise--Shinder  \cite[Theorem 4.1.2]{nicaise2019motivic}.

       Finally, suppose that $\mc{P}$ is property (iii). The geometric fiber $X_{\cl{\kappa(\!(t)\!)}}$ is universally $CH_0$-trivial, and hence, as a special case of Colliot-Th\'el\`ene--Pirutka \cite[Theorem 1.14]{colliot2016hypersurfaces} (see \Cref{ch0-trivial-specialization} below), the $k$-variety $X_\kappa$ is universally $CH_0$-trivial, that is, $s_0$ satisfies $\mc{P}$. Thus property (iii) satisfies the specialization condition.        
        
       In all three cases (i)-(iii), we have proved that $\mc{P}$ satisfies the rigidity and specialization conditions. The conclusion follows from \Cref{thm.main1}.
	\end{proof}

\subsection{Degenerations to mildly singular varieties}

	Finally, we consider the specialization method of Voisin, later generalized by Colliot-Th\'el\`ene--Pirutka. The terminology \enquote{specialization method} is standard, and refers to the technique of proving stable irrationality of a very general member of a class of smooth projective varieties by degenerating it to a suitable mildly singular variety (in the original setup of Voisin, to varieties with at most quadratic singularities). It should not be confused with the term \enquote{specialization condition} used in this paper. 
    
 The specialization condition that is relevant to us is provided by the following result of Colliot-Th\'el\`ene and Pirutka.
 
 \begin{defin} \label{def:CH0-relative}
     Let $k$ be a field, and let $p\colon V\to W$ be a proper morphism of $k$-varieties. We say that $p$ is \emph{universally $CH_0$-trivial} if, for every field extension $F/k$, the pushforward map $p_*\colon CH_0(V_F)\to CH_0(W_F)$ is an isomorphism.
 \end{defin}

Note that~\Cref{def:CH0-relative} is a relative version of~\Cref{rationality-definitions}(iii): a proper $k$-variety $V$ is universally $CH_0$-trivial 
if and only if the structure morphism $V \to \Spec(k)$ is universally $CH_0$-trivial.

 \begin{thm}[Colliot-Th\'el\`ene--Pirutka]\label{ch0-trivial-specialization}
		Let $R$ be a discrete valuation ring with fraction field $K$ and algebraically closed residue field $k$, and let $\overline{K}$ be an algebraic closure of $K$. Let $f\colon X\to \on{Spec}(R)$ be a faithfully flat, proper morphism with geometrically integral fibers. Assume that the special fiber $Y \coloneqq X \times_R k$ admits a desingularization $\nu\colon \widetilde{Y} \to Y$ such that the morphism $\nu$ is universally $CH_0$-trivial, and that the geometric generic fiber $X_{\overline{K}} := X \times_A \overline{K}$ admits a desingularization $\widetilde{X} \to X_{\overline{K}}$. If $\widetilde{X}$ is universally $CH_0$-trivial, then so is $\widetilde{Y}$.
	\end{thm}

	\begin{proof}
		See \cite[Theorem 1.14]{colliot2016hypersurfaces}.
	\end{proof}
	
	\begin{thm}[Voisin, Colliot-Th\'el\`ene--Pirutka]\label{voisin-specialization}
		Let $k$ be an uncountable algebraically closed field, let $S$ be an integral $k$-scheme of finite type, let $\eta\in S$ be the generic point of $S$, and let $f\colon X\to S$ be a flat proper morphism with geometrically integral fibers. Suppose that the $k(\eta)$-variety is smooth, and that there exist a geometric point $s_0\colon \on{Spec}(k_0)\to S$ and a desingularization $\nu_0\colon Z\to X_{s_0}$ such that $\nu_0$ is universally $CH_0$-trivial and $Z$ is not universally $CH_0$-trivial. Then the set of $s\in S(k)$ such that $X_s$ is smooth projective and not universally $CH_0$-trivial is very general.
	\end{thm}
	
	\begin{proof}
    Let $\mc{P}$ be the following property of geometric morphisms over $k_0$: for every algebraically closed field $k'$ containing $k_0$ and every $k'$-morphism $V\to W$, the pair $(k', V\to W)$ satisfies $\mc{P}$ if and only if $W=\on{Spec}(k') $ and $V$ admits a desingularization $\nu\colon \tilde{V}\to V$ such that $\nu$ is universally $CH_0$-trivial and $\tilde{V}$ is universally $CH_0$-trivial. A variant of the proof of \Cref{rationality-properties-rigidity} shows that $\mc{P}$ satisfies the rigidity condition; we leave this as an exercise for the reader. By assumption, the pair $(k(s_0),f_{s_0})$ does not satisfy $\mc{P}$. 
    
	By \cite[Proposition 7.1.9]{ega2}, there exist a discrete valuation ring $R$ and a morphism $\varphi\colon \on{Spec}(R)\to S$ sending the generic point $\on{Spec}(R)$ to $\eta$ and the closed point of $\on{Spec}(R)$ to the image of $s_0$. We let $f_R\colon X_R\to \on{Spec}(R)$ be the base change of $f$ along $\varphi$. Since the generic fiber of $f$ is smooth, there exists a dense open subscheme $U\subset S$ such that $f_U\colon X_U\to U$ is smooth. By \Cref{ch0-trivial-specialization}, property $\mc{P}$ satisfies the specialization condition with respect to $(U,f_U)$ and $(\on{Spec}(R),f_R)$. The conclusion now follows from \Cref{thm.main2}.
	\end{proof}
	
	\begin{rmk}
		Voisin \cite{voisin2015unirational} applied her specialization method to a suitable projective family $f\colon X\to S$ to prove that the desingularization of a very general quartic double solid with at most seven nodes does not admit a decomposition of the diagonal, and hence it is not retract rational. Colliot-Th\'el\`ene--Pirutka \cite{colliot2016hypersurfaces} applied \Cref{voisin-specialization} to prove that a very general quartic threefold does not admit a decomposition of the diagonal, and hence it is not retract rational. 
	\end{rmk}

\section{Geometric rational sections}
\label{sect.rat-sections}

\begin{defin}\label{defin:rational-section}
Let $k$ be a field, and let $f\colon X \to  Y$ be a morphism of $k$-schemes of finite type. We say that $f$ \emph{has a rational section} if there exist an open dense subscheme $U \subset Y$, and a section of $f$ over $U$, namely, a commutative diagram
  \[
\xymatrix{
& X \ar[d]^f \\
U \, \ar@{^{(}->}[r] \ar[ur] & Y.
}
\] 
We say that $f$ \emph{has a geometric rational section} if $f_{\overline{k}}\colon X_{\overline{k}} \to  Y_{\overline{k}}$ has a rational section, where $\overline{k}$ is the algebraic closure of $k$.
\end{defin}

\begin{rmk}
    \begin{enumerate}
        \item In the definition above, the open dense subscheme $U$ is only assumed to be topologically dense, not scheme-theoretically dense.
        \item If $Y$ is reduced and $\eta_{1}$, \dots,~$\eta_{r}$ are the generic points of its irreducible components, then $f$ has a rational section if and only if $\Spec\kappa(\eta_{i}) \stackrel{\; \eta_i \;}{\longrightarrow}  Y$ lifts to a morphism $\Spec k(\eta_{i}) \to   X$ for each $i = 1, \ldots, r$.
        \item If $f\colon X \to  Y$ as above has a rational section, it is clear that for any extension $k'/k$, the morphism $f_{k'}\colon X_{k'} \to  Y_{k'}$ has a rational section.
    \end{enumerate} 
\end{rmk}

\begin{lemma}\label{lem1}
Let $S$ be a noetherian scheme, $f\colon T \to  S$ be a morphism of finite type, and $T' \subset T$ be an open subset. Then the subset
   \[
   S' \coloneqq \{s \in S\mid T'_{s}\text{ is dense in }T_{s}\} \subset S
   \]
is constructible. Furthermore, if $T'$ is dense in $T$, then $S'$ is dense in $S$.
\end{lemma}

\begin{proof}
For each $d \geqslant 0$ we denote by $T^{d}$ the locus of points $t \in T$ such that $\dim_{t}T_{f(t)} = d$. By semicontinuity of the fiber dimension we have that $T^{d}$ is a locally closed subset of $T$; also, $T^{d} \neq \emptyset$ for finitely many values of $d$, because $S$ is noetherian.

Set $V \eqdef T \setminus T'$; if $s \in S$, $T'$ fails to be dense in the fiber $T_{s}$ if and only if there is a point $t \in V_{s}$ such that $\dim_{t}V_{s} = \dim_{t}T_{s}$; hence, the complement $S\setminus S'$ is the union of the images of the intersections $V^{d} \cap T^{d} \subset T$. From Chevalley's theorem it follows that $S\setminus S'$ is constructible, hence $S'$ is constructible.

If $T'$ is dense in $T$, then $S'$ contains the generic points of all the irreducible components of $S$, hence it is dense in $S$.
\end{proof}

\begin{lemma}\label{prop:properties-grs}
Let $k$ be a field, and let $f\colon X \to  Y$ be a morphism of $k$-schemes of finite type. The following are equivalent.
\begin{enumerate}

\item The morphism $f$ has a geometric rational section.

\item There exists a finite extension $k'/k$ such that $f_{k'}\colon X_{k'} \to  Y_{k'}$ has a rational section.

\item There exists some extension $k'/k$ such that $f_{k'}\colon X_{k'} \to  Y_{k'}$ has a rational section.

\end{enumerate}
\end{lemma}

\begin{proof}
(1) $\Longrightarrow$ (2). It suffices to note that a rational section of $f_{\cl{k}}$ is defined over a finite intermediate field $k \subset k' \subset \overline{k}$.

(2) $\Longrightarrow$ (3). Obvious. 

(3) $\Longrightarrow$ (1). For this we can extend $k'$, and assume $\overline{k} \subset k'$; we can also assume $k = \overline{k}$. The rational section $Y_{k'} \to  X_{k'}$ is defined over a finitely generated extension of $k$; hence we can assume that $k'$ is finitely generated over $k$. Suppose that $U \to  X_{k'}$ is a rational section of $f_{k'}$; there exists a finitely generated $k$-subalgebra $R\subset k'$, an open subscheme $V$ of $Y_{R}$ such that $V_{k'}$, and a section $V \to  X_{R}$ of $f_R$. Since $V$ is dense in $Y_{R}$, by \Cref{lem1} applied to $V \subset Y_{R}$ we can localize $R$ and assume that each fiber of $V \to  \Spec(R)$ is dense in the corresponding fiber of $Y_{R} \to  \Spec(R)$. Since $k$ is algebraically closed, we can restrict the rational section $V \to  X_{R}$ to a $k$-rational point in $s \in \Spec(R)$ and conclude that there is a rational section $V_{s} \to  X$ of $f$, as desired. 
\end{proof}

We are now ready for the main theorem of this section.

\begin{thm} \label{thm.rat-sections} Let $k$ be a field, let $S$ be a $k$-scheme of finite type, let $Y\to S$ be a flat morphism of finite type, and let $f\colon X\to Y$ be a proper morphism. Then there exist a countable collection of closed subschemes $\{S_j\}_{j\geqslant 1}$ of $S$ with the following property: for every geometric point $s$ of $S$,
the morphism $f_{s}$ has a rational section if and only if $s$ lies in $S_j$ for some $j \geqslant 1$.
\end{thm} 

\begin{proof} Let $\mathcal{P}$ be the property of geometric morphisms of having a rational section. In view of \Cref{thm.main1}, it suffices to show that $\mathcal{P}$ satisfies the rigidity condition, as well as the specialization condition with respect to $(S,f)$.
Rigidity is an immediate consequence of \Cref{prop:properties-grs}; we will focus on verifying the specialization condition for the remainder of the proof.

By \Cref{reduce-to-dvr} we can assume $S = \Spec(R)$. 
Denote the the fraction field of $R$ by $K$
and its residue field of $R$ by $\kappa$. By our assumption
$f$ has a section over $\Spec \cl{K}$. Then there exists a finite extension $K'$ of $K$ such that $f_{K'}\colon X_{K'} \to  Y_{K'}$ has a rational section. Choose a discrete valuation ring $R' \subset K'$ dominating $R\subset K$; then again by \Cref{prop:properties-grs} we may replace $R$ with $R'$, and so assume that $f_{K}\colon X_{K} \to  Y_{K}$ has a rational section. We will show that $f_{\kappa}\colon X_{\kappa} \to  Y_{\kappa}$ has a rational section. By hypothesis, $Y_{\kappa}$ is reduced.

If $Y_{\kappa}=\emptyset$, then $X_{\kappa}=\emptyset$ and $f_{\kappa}\colon X_{\kappa}\to Y_{\kappa}$ is the identity map, which has a section and hence a rational section. We may thus assume that $Y_{\kappa}\neq\emptyset$. Let $\eta\in Y_{\kappa} \subset Y$ be the generic point of an irreducible component, set $A \coloneqq \mathcal{O}_{Y,\eta}$, and let $L$ be the fraction field of $A$. We need to show that $\Spec \kappa(\eta) \to  Y$ lifts to $X$. Let $t \in R$ be a uniformizing parameter: since $Y \to  \Spec(R)$ is flat we have that $\mathfrak{m}_{A} = tA$, and $t$ is not a zero-divisor on $A$. Hence $A$ is a discrete valuation ring, and the morphism $\Spec A \to  X$ sends the generic point $\Spec L$ of $\Spec A$ into the generic point of a component of $X_{K}$. Hence $\Spec L \to  Y_{K} \subset Y$ lifts to $\Spec L \to  X_{K} \subset X$. The desired conclusion now follows from the valuative criterion for properness applied to $X \to  Y$.
\end{proof}

\begin{rmk}
    The assumption that $Y$ is flat over $S$ in \Cref{thm.rat-sections} may not be removed. Indeed, let $S=\P^1_k$, let $Y=X$ be the $S$-scheme given by the disjoint union of $S$ and $T=\P^1_k$, where $T\to S$ is the constant map sending $T$ to the $k$-point $0 \in \mathbb P^1 = S$. Now let $f\colon X\to Y$ be the projective morphism given by identity on $S$ and by the map $(u:v)\mapsto (u^2:v^2)$ on $T$. Observe that $Y$ is not flat over $S$. For every geometric point $s$ of $S$, the fiber $f_s$ has a geometric section if and only if the image of $s$ is not equal to $0\in S$. Thus the conclusion of \Cref{thm.rat-sections} fails in this example.
\end{rmk}

\section{Examples}
\label{sect.examples}

\subsection{Variation of Witt indices in a family of quadrics}

Recall that by the Witt Decomposition Theorem, every non-degenerate quadratic form $q$ over a field $K$ of characteristic different from $2$ can be written, in a unique way, as a direct sum $q \simeq q_{\rm an} \oplus \langle -1, 1 \rangle \oplus \ldots \oplus \langle -1, 1 \rangle$, where
$q_{\rm an}$ is an anisotropic form over $K$ and $\langle -1, 1 \rangle$ is the hyperbolic plane. The number of hyperbolic planes in this decomposition is called the Witt index of $q$. We will denote it by $w_0(q)$.

Let $Q_0$ be the projective quadric cut out by $q$. Let $K_1 = K(Q)$ be the function field of $Q$. The first Witt index $w_1(q)$ is defined as $w_0(q_{K_1})$, where $q_{K_1}$ is obtained by base-changing $q$ to $K_1$. Proceeding recursively, we define $K_0 = K$ and, for every integer $m\geqslant 1$, $K_m = K_{m-1}(Q)$ to be the function field of $Q$ over $K_{m-1}$, and $w_m(q) = w_0(q_{K_m})$.
Note that the sequence $w_0(q), w_1(q), w_2(q), \ldots$ is strictly increasing until it reaches its maximal possible value $\lfloor \dim(q)/2 \rfloor$, and it stabilizes at $\lfloor \dim(q)/2 \rfloor$ after that. For this reason we will always assume that $m \leqslant \lfloor \dim(q)/2 \rfloor -1$. Since $q$ is non-degenerate, the quadric hypersurface $Q \subset \mathbb{P}^{n-1}$ is absolutely irreducible for any $n \geqslant 3$. Thus when $n \geqslant 3$, $K_m$ is, indeed, a field for every $m \geqslant 1$. For $n = 2$, $K_1$ may not be a field, but this will not be a problem for us, because in this case we stop at $m = \lfloor n/2 \rfloor -1 = 0$.

Note also that the Witt indices $w_m(q)$ do not change if we multiply $q$ by a non-zero constant from $K$. Hence, $w_m(q)$ may be viewed as numerical invariants of the quadric $Q$ it cuts out in $\mathbb P^{\dim(q) - 1}$. We will thus sometimes write $w_m(Q)$ in place of $w_m(q)$.

We will be interested in how the Witt index $w_m$ varies in
a family of quadrics or rank $n$. Here $m$ and $n$ are fixed.
 By a family of quadrics we mean a quadric bundle $Q \to Y$, i.e., the zero locus of a line bundle-valued quadratic form $\phi \colon \mathcal{E} \to \mathcal{L}$ over $Y$ in 
$\mathbb P(\mathcal{E})$.

\begin{prop} \label{ex.quadric} Let $k$ be a field $k$ of characteristic different from $2$, let $S$ be a $k$-variety, let $Y \to S$ be a flat morphism of finite type, and let $Q \to Y$ be a quadric bundle over $Y$. Fix integers $m \geqslant 0$ and $d \geqslant 1$. Then 
there exists a countable collection of closed subschemes $\{S_j\}_{j\geqslant 1}$ of $S$ such that
the generic $m$-th Witt index of the fiber $Q_s \to Y_s$ over a geometric point $s$ of $S$ is $\geqslant d$ if and only if $s$ lies in $S_j$ for some $j \geqslant 1$. 
\end{prop}

Here by the generic $m$-th Witt index of $Q_s \to Y_s$ being $\geqslant d$ we mean that if $\eta\in Y_s$ is the generic point of some irreducible component of $Y_s$, then we view $Q_{\eta}$ is a quadric hypersurface over $k(\eta)$, and we ask that its $m$th Witt index $w_m(Q_{\eta})$ should be $\geqslant d$.

\begin{proof} First assume that $m = 0$. In this case we apply \Cref{thm.rat-sections} to the morphism $f \colon X \to Y$, where $X = \Gr^{\rm iso}(d, Q)$ is the isotropic Grassmannian scheme parametrizing $d$-dimensional linear subspaces 
contained in the fibers of $Q$. Indeed, the generic Witt index of the fiber $Q_s \to Y_s$ is $\geqslant d$ if and only if $f_s \colon X_s \to Y_s$ has a rational section. 

Now assume that $m = 1$. 
In this case we replace $Y$ by $Y_1 := Q$ and $Q$ by $Q_1 := Q \times_Y Y_1$. Then for every geometric point $s \colon \Spec(k) \to S$, 
first Witt index $Q_s \to Y_s$ is $\geqslant d$ if and only if the 
generic Witt index of $(Q_1)_s \to (Y_1)_s$ is $\geqslant d$, and we can appeal to the case, where $m = 0$.

For $m \geqslant 2$, we proceed recursively, replacing $Y_{m-1}$ by
$Q_{m-1}$ and $Q_{m-1}$ by $Q_m := Q_{m-1} \times_{Y_{m-1}} Y_m$.
\end{proof}

\subsection{Variation of the Schur index in a family of central simple algebra}
\label{sect.index}
Recall that by Wedderburn's theorem, every central simple algebra $A$ over a field $K$ can be uniquely written in the form $A \simeq \Mat_s(D)$, where $D$ is a division algebra over $K$. The degree of $D$ over $K$ is called the Schur index of $A$.

\begin{prop} \label{ex.azumaya} Let $k$ be a field, let $Y \to S$ be a flat morphism of $k$-varieties, and let $\mathcal{A}$ be an Azumaya algebra of degree $n$ over $Y$. Fix a positive integer $d$.
Then there exists a countable collection of closed subschemes $\{S_j\}_{j\geqslant 1}$ of $S$ such that the generic Schur index of the fiber $\mathcal{A}_s$ over a geometric point $s$ of $S$ divides $d$ if and only if $s$ lies in $S_j$ for some $j \geqslant 1$. 
\end{prop}

Here $\mathcal{A}_s$ is an Azumaya algebra of degree $n$ over $Y_s$. By the generic Schur index of $\mathcal{A}_s$ dividing $d$, we mean that if
$\eta \in Y_s$ is a generic point of $Y_s$, then we view $A_{\eta}$ as a central simple algebra of degree $n$ over $k(\eta)$, and we ask that Schur index of $\mathcal{A}_{\eta}$ over $K$ should divide $d$. (This should be true for every irreducible component of $Y_s$.)

\begin{proof} We apply \Cref{thm.rat-sections} to the morphism $f \colon X \to Y$, where $X =\SB(d, \mc{A})$ is the generalized Severi-Brauer variety parametrizing 
right ideals of dimension $dn$ in the fibers of $\mathcal{A}$.  
Indeed, the generic Schur index of the fiber $\mc{A}_s$ divides $d$ if and only if $f_s \colon X_s \to Y_s$ has a rational section. 
\end{proof}

\section{Properties of torsors over fields}
\label{sect.properties-of-torsors}

Let $k$ be a field, and let $G$ be a $k$-group. For every $k$-scheme $V$, there is a natural equivalence between the groupoid of $G$-torsors over $V$ and the groupoid $B_kG(V)$ of morphisms $V \to B_kG$. 

\begin{defin}\label{defin:property-torsors}
    Let $k$ be a field, and let $G$ be a $k$-group. 
    \begin{itemize}
        \item[(1)] By definition, a \emph{property of $G$-torsors over fields} is a property $\mc{P}$ of geometric morphisms $(k', \on{Spec}(K) \to B_{k'}G)$ over $k$ (cf. \Cref{defin:geometric-morphism}), where $k'$ is algebraically closed field and $K/k'$ is a field extension. That is, if a pair $(k',\varphi\colon V\to W)$ as in \Cref{defin:geometric-morphism} satisfies $\mc{P}$, then $V=\on{Spec}(K)$ for some field extension $K/k'$ and $W=B_{k'}G$. We say that a $G$-torsor over $K$ satisfies $\mc{P}$ if the corresponding pair $(k', \on{Spec}(K) \to B_{k'}G)$ satisfies $\mc{P}$.
        \item[(2)] A \emph{property of $G$-torsors over varieties} is a property $\mc{P}$ of geometric morphisms $(k', V \to B_{k'}G)$ over $k$, where $V$ is a $k'$-variety. We say that a $G$-torsor over $V$ satisfies $\mc{P}$ if the corresponding pair $(k', V \to B_{k'}G)$ satisfies $\mc{P}$. 
        \item[(3)] A property $\mc{P}$ of $G$-torsors over fields determines a property of $G$-torsors over varieties: for every algebraically closed field $k'$ containing $k$ and every $k'$-variety $V$, a $G$-torsor $E\to V$ satisfies $\tilde{\mc{P}}$ if and only if for every generic point $\eta\in V$, the restriction $E_{k'(\eta)}\to \on{Spec}(k'(\eta))$ satisfies $\mc{P}$. If the $G$-torsor $E\to V$ satisfies $\mc{P}$, we say that it \emph{satisfies $\mc{P}$ generically}. We refer to $\tilde{\mc{P}}$ as \emph{the property of satisfying $\mc{P}$ generically}.
    \end{itemize}
\end{defin}

\begin{rmk} All properties of $G$-torsors over varieties that we will consider in this paper will be of the form $\tilde{\mc{P}}$ for some property $\mc{P}$ of $G$-torsors over fields. In particular, they will only depend on the fibers of the $G$-torsor $E \to V$ over the generic points of the $k'$-variety $V$, that is, on the pullback to $\on{Spec}(k'(V))$.
\end{rmk}

\begin{defin}\label{defin:specialization-torsors}
    Let $k$ be a field, let $G$ be a $k$-group, let $\mc{P}$ be a property of $G$-torsors over fields, let $S$ be a locally noetherian $k$-scheme, let $X$ be an $S$-scheme, let $\alpha\colon E\to X$ be a $G$-torsor. Assume that $\mc{P}$ satisfies the rigidity condition. We will say that $\mc{P}$ \emph{satisfies the specialization condition with respect to $(S,\alpha)$} if it satisfies the specialization condition with respect to $(S,f)$ (in the sense of \Cref{defin:specialization}), where $f\colon X\to [S/G]$ is the classifying morphism of $\alpha$.
\end{defin}

\begin{rmk}
    In order to fit properties of torsors into the general framework of properties of morphisms given in \Cref{sect.morphisms}, we have formally defined them in \Cref{defin:property-torsors}  via their classifying morphisms. Note that the map $E \to X$ of \Cref{defin:property-torsors}(3) does not fit into this framework directly because a $G$-torsor is not just a morphism but a morphism with extra structure. This extra structure, namely a $G$-action on $E$, is captured by the classifying morphism. In practice, it is more natural to work with the torsor than with its classifying morphism; see for instance Examples~\ref{ex.ed} and \ref{ex.rd} below.
\end{rmk}

 \Cref{thm.main1} specializes to properties of $G$-torsors over varieties as follows.

\begin{thm}\label{thm.main.torsors}
    Let $k$ be a field, let $G$ be a $k$-group, let $\mc{P}$ be a property of $G$-torsors over varieties, let $S$ be a $k$-scheme of finite type, let $\alpha\colon E \to X$ be a $G$-torsor of $S$-schemes of finite type. Suppose that $\mc{P}$ satisfies the rigidity condition, as well as the specialization condition with respect to $(S,\alpha)$.  Then there exists a countable family $\{S_j\}_{j\geqslant 1}$ of closed subschemes of $S$ such that the following are equivalent for any geometric point
    $s$ of $S$:
    
    \smallskip (i) the $G$-torsor $\alpha_s \colon E_s \to X_s$ over $s$ satisfies $\mc{P}$, and 
    
    \smallskip (ii) $s$ lies in $S_j$ for some $j \geqslant 1$.    
\end{thm}

\begin{proof}
    Apply \Cref{thm.main1} to the morphism $f\colon X\to [S/G]$ which classifies $\alpha$.
\end{proof}

In the setting of properties of torsors over fields, the following refinement of the valuative criterion \Cref{reduce-to-dvr} is available.

\begin{lemma}\label{reduce-to-dvr-torsors}
    Let $k$ be a field, let $G$ be a $k$-group, let $\mc{P}$ be a property of $G$-torsors over fields which satisfies the rigidity condition, let $S$ be a locally noetherian $k$-scheme, let $\alpha\colon E\to X$ be a morphism of algebraic $S$-stacks. The following statements are equivalent.
    \begin{itemize}
        \item[(1)] Property $\tilde{\mc{P}}$ (cf. \Cref{defin:property-torsors}(3)) satisfies the specialization condition with respect to $(S,\alpha)$.
        \item[(2)] For every morphism $\on{Spec}(R)\to S$, where $R$ is a complete discrete valuation ring $R$, property $\tilde{\mc{P}}$ satisfies the specialization condition with respect to $(\on{Spec}(R),\alpha_R)$.
    \end{itemize}
    Suppose further that $X$ is a faithfully flat $S$-scheme of finite type with geometrically integral $S$-fibers. Then (1) and (2) are also equivalent to the following statement about $\mc{P}$.
    \begin{itemize}
        \item[(3)] For every commutative diagram with cartesian squares
        \[
\xymatrix{
T \ar[r] \ar[d] & E_R \ar[r] \ar[d]^{\alpha_R} & E \ar[d]^{\alpha} \\
\operatorname{Spec}(A) \ar[r]^{\varphi} & X_R \ar[r] \ar[d] & X \ar[d] \\
& \operatorname{Spec}(R) \ar[r] & S
}
\]
        where $R$ and $A$ are complete discrete valuation rings and $\varphi$ sends the generic (resp. closed) point of $A$ to the generic point of the generic (resp. closed) fiber of $X_R\to\on{Spec}(R)$, property $\mc{P}$ satisfies the specialization condition with respect to $(\on{Spec}(A),T\to\on{Spec}(A))$.
    \end{itemize}
\end{lemma}

\begin{proof} The equivalence between (1) and (2) is a special case of \Cref{reduce-to-dvr}. 

Consider a diagram as in the statement of (3). Let $K$ be the fraction field of $A$, and let $x_0\in X$ and $s_0\in S$ be the images of the generic point of $\on{Spec}(A)$. By assumption, $x_0$ is the generic point of a fiber of $X\to S$. We have a commutative diagram with cartesian squares
\[
\xymatrix{
T_K \ar[r] \ar[d] & E_{k(x_0)} \ar[r] \ar[d] & E_{s_0} \ar[d] \\
\operatorname{Spec}(A) \ar[r] & \operatorname{Spec}(k(x_0)) \ar[r] & X_{s_0}.
}
\]
Thus, letting $\cl{s}_0$ be a geometric point of $S$ lying over $s_0$, by the rigidity condition a geometric generic fiber of $T\to\on{Spec}(A)$ satisfies $\mc{P}$ if and only if $E_{\cl{s}_0}\to X_{\cl{s}_0}$ satisfies $\mc{P}$ generically (that is, satisfies $\tilde{\mc{P}}$). Similarly, letting $s_1\in S$ be the image of the closed point of $\on{Spec}(A)$ and $\cl{s}_1$ be a geometric point of $S$ lying over $s_1$, a geometric special fiber of $T\to\on{Spec}(A)$ satisfies $\mc{P}$ if and only if $E_{\cl{s}_1}\to X_{\cl{s}_1}$ satisfies $\tilde{\mc{P}}$. 

Observe that $s_1$ is a specialization of $s_0$. Therefore, the previous paragraph shows that the specialization condition for $\tilde{\mc{P}}$ with respect to $(S,\alpha)$ implies the specialization condition for $\mc{P}$ with respect to $(\on{Spec}(A),T\to\on{Spec}(A))$, so that (1) implies (3). The previous paragraph also shows that, in order to prove that (3) implies (2), it suffices to extend every commutative diagram with cartesian squares
\[
\xymatrix{
E_R \ar[r] \ar[d]^{\alpha_R} & E \ar[d]^{\alpha} \\
X_R \ar[r] \ar[d] & X \ar[d] \\
\operatorname{Spec}(R) \ar[r] & S,
}
\]
where $R$ is a complete discrete valuation ring, to a commutative diagram as in the statement of (3). For this, we let $A'$ be the local ring of the generic point of the closed fiber of $X_R\to\on{Spec}(R)$, viewed as a point of $X_R$.    Since the closed fiber of $X_R\to \on{Spec}(R)$ is geometrically integral, the maximal ideal of $A'$ is generated
by the image $\nu$ of a uniformizer of $R$. Since $X_R\to \on{Spec}(R)$ is flat and surjective (the empty scheme is not integral) it is faithfully flat. It follows that $\nu$ is not a zero-divisor in $A'$, and so $A'$ is a discrete valuation ring. We have a natural map $\on{Spec}(A')\to X_R$. It now suffices to let $A$ be the completion of $A'$, and to let $\varphi$ be the composite $\on{Spec}(A)\to \on{Spec}(A')\to X_R$.
\end{proof}

\begin{example} \label{ex.ed}
    Let $k$ be a field, let $G$ be a $k$-group, and let $n$ be a non-negative integer. 
    Consider the following property $\mc{P}_n$ of $G$-torsors over fields: for every field extension $K/k'$, where $k'$ is an algebraically closed field containing $k$, a $G$-torsor $E\to \on{Spec}(K)$ satisfies $\mc{P}_n$ if and only if its essential dimension over $k'$ is  $\leqslant n$.
    Under a mild assumption on the characteristic of $k$\footnote{More specifically, we require that $G$ should be in good characteristic; see \Cref{def.good-characteristic}.}, the property of satisfying $\mc{P}_n$ generically satisfies the rigidity condition by \cite[Lemma 2.2]{reichstein2022behavior}. 
    For the specialization condition, by \Cref{reduce-to-dvr-torsors} it suffices to show that, for every discrete valuation ring $A$ containing $k$ and for every $G$-torsor $E\to \on{Spec}(A)$, the essential dimension of a geometric generic fiber of $E\to \on{Spec}(A)$ is greater than or equal to the essential dimension of a geometric special fiber of $E\to \on{Spec}(A)$. This is proved in \cite[Theorem 6.4]{reichstein2022behavior}.

Therefore \Cref{thm.main.torsors} applies, and shows that in a family of $G$-torsors $E \to X$ over a $k$-scheme of finite type $S$, where the morphism $X \to S$ is flat with geometrically integral fibers, the geometric points $s$ of $S$ such that the generic fiber $E_{k(X_s)}\to \on{Spec}(k(X_s))$ has essential dimension $\leqslant n$ belong to countable union of closed subschemes. See \cite[Theorem 1.2]{reichstein2023behavior}, where the more general case of  primitive generically-free $G$-varieties is considered.
\end{example}

\begin{example} \label{ex.rd}
Let $k$ be a field, let $G$ be a $k$-group, and let $n$ be a non-negative integer.  The assertion of the previous example continues to hold if essential dimension is replaced by resolvent degree. More precisely, let 
$\mc{P}_n$ be the following property of $G$-torsors over fields: for every field extension $K/k'$, where $k'$ is an algebraically closed field containing $k$, a $G$-torsor $E\to \on{Spec}(K)$ satisfies $\mc{P}_n$ if and only if its resolvent degree over $k'$ is $\leqslant n$.\footnote{For the definition and basic properties of the resolvent degree of a $G$-torsor over a field $K/k$, where $G$ is an algebraic group over $k$, we refer the reader to \cite{reichstein-hp13}. In the case where $G$ is a finite group and $\Char(k) = 0$ an earlier and more geometric treatment can be found in Farb and Wolfson~\cite{farb-wolfson}.} Let $\tilde{\mc{P}}_n$ be the property of satisfying $\mc{P}_n$ generically. We claim that property $\tilde{\mc{P}}_n$ satisfies the rigidity and specialization conditions. For the case $n=0$, see \Cref{ex.split}. Here we will assume that $n\geqslant 1$.

\smallskip

The rigidity condition for property $\tilde{\mathcal{P}}_n$ asserts that, for every extension $k\subset k'\subset k''$, where $k'$ and $k'$ are algebraically closed fields, every integral $k'$-variety $V$, and every $\tau\in H^1(k'(V),G)$, we have 
\[\rd_{k'}(\tau) \leqslant n\quad \iff \quad \rd_{k''}(\tau_{k''(V)}) \leqslant n.\] For the inequality $\rd_{k'}(\tau) \geqslant \rd_{k''}(\tau_{k''(V)})$, see 
the proof of~\cite[Proposition 8.3(a)]{reichstein-hp13}, in particular, formula (8) there. Conversely, the inequality $\rd_{k''}(\tau_{k''(V)}) \leqslant
\max \{ \rd_{k'}(\tau), \, 1 \}$ is established in the proof of~\cite[Proposition 9.1]{reichstein-hp13}; in particular, see formula (12) there. Note that Propositions 8.3 and 9.1 in \cite{reichstein-hp13} are proved in a more general setting, where $\tau\in \mc{F}(K)$ for some field extension $K/k$ and some functor $\mathcal{F}$ from the category of field extensions of $k$ to the category of sets, under the assumption that $\mathcal{F}$ satisfies conditions (7) and (11) of \cite{reichstein-hp13}. These results can be applied to the functor $\mathcal{F} := H^1(-, G)$ because this functor satisfies conditions (7) (obvious) and (11); see \cite[Lemma 12.2]{reichstein-hp13}. Thus property $\tilde{\mc{P}}_n$ satisfies the rigidity condition. 

For the specialization condition, by \Cref{reduce-to-dvr-torsors} it suffices to show that, for every discrete valuation ring $A$ containing $k$ and for every $G$-torsor $E\to \on{Spec}(A)$, the resolvent degree of a geometric generic fiber of $E\to \on{Spec}(A)$ is greater than or equal to the resolvent degree of a geometric special fiber of $E\to \on{Spec}(A)$.  This is established in the course of the proof of \cite[Proposition 13.1]{reichstein-hp13}; in particular, see formula (15) there. 

Thus \Cref{thm.main.torsors} applies and shows that, for a $k$-scheme of finite type, a flat $S$-scheme of finite type with geometrically integral $S$-fibers, a $G$-torsor $E \to X$, and an integer $n \geqslant 1$, the images of the geometric points $s$ of $S$ such that the generic fiber of $E_s\to X_s$ has resolvent degree $\leqslant n$ belong to countable union of closed subschemes.
\end{example}

We conclude this section with the following variant of \Cref{thm.main2} for $G$-torsors over fields, which will be useful in \Cref{sect.strongly-unramified}.

\begin{thm}\label{thm.main2-torsors} 
    Let $k$ be an algebraically closed field, let $G$ be a $k$-group, let $\mc{P}$ be a property of $G$-torsors over fields, let $S$ be an irreducible $k$-scheme of finite type, let $X\to S$ be a faithfully flat morphism of finite type such that a general fiber of $X\to S$ is geometrically integral, let $\alpha\colon E\to X$ be a $G$-torsor, and let $s_0\in S(k)$ be such that $X_{s_0}$ is reduced and $\alpha_{s_0}$ does not satisfy $\mc{P}$ generically. Assume that $\mc{P}$ satisfies the rigidity condition, as well as the specialization condition with respect to all $G$-torsors $T\to\on{Spec}(A)$, where $A$ is a complete discrete valuation ring containing $k$. Then the set of $s\in S(k)$ such that $\alpha_s$ does not satisfy $\mc{P}$ generically is very general.
\end{thm}

\begin{proof}
Let $\tilde{\mc{P}}$ be the property of satisfying $\mc{P}$ generically; see \Cref{defin:property-torsors}(3). In view of \Cref{thm.main2}, it suffices to show that there exist:
\begin{itemize}
        \item[(i)] a Zariski dense open subscheme $U\subset S$ such that $\tilde{\mc{P}}$ satisfies the specialization condition with respect to $(U,\alpha_U)$, and
        \item[(ii)] a morphism $\varphi\colon \on{Spec}(R)\to S$, where $R$ is a discrete valuation ring, sending the closed point of $\on{Spec}(R)$ to $s_0$ and the generic point of $\on{Spec}(R)$ to a point of $U$, such that $\tilde{\mc{P}}$ satisfies the specialization condition with respect to $(\on{Spec}(R),\alpha_R)$.
\end{itemize}
    By assumption, there exists a dense open subscheme $U\subset S$ such that every geometric fiber of the restriction $X_U\to U$ is geometrically integral. By \Cref{reduce-to-dvr-torsors} and the assumption that $\mc{P}$ satisfies the specialization condition with respect to $G$-torsors over complete discrete valuation rings containing $k$, we deduce that property $\tilde{\mc{P}}$ satisfies the specialization condition with respect to $(U,\alpha_U)$, that is, (i) holds. 
    
    By \cite[Proposition 7.1.9]{ega2}, there exist a discrete valuation ring $R$ and a morphism $\varphi\colon \on{Spec}(R)\to S$ which sends the closed point of $\on{Spec}(R)$ to $s_0$ and the generic point of $\on{Spec}(R)$ to the generic point of $U$. We may assume that $R$ is complete. By assumption, there exists one irreducible component $X'\subset X_{s_0}$ such that the restriction of $E\to X_{s_0}$ to $X'$ does not satisfy $\mc{P}$ generically. We may remove from $X$ the union of the other irreducible components of $X_{s_0}$ (a closed subset of $X$) to assume that $X_{s_0}=X'$ is integral. By \Cref{reduce-to-dvr-torsors} applied to $S=\on{Spec}(R)$ and the specialization condition for $\mc{P}$ with respect to $G$-torsors over complete discrete valuation rings containing $k$, we conclude that $\tilde{\mc{P}}$ satisfies the specialization condition with respect to $(\on{Spec}(R),\alpha_R)$, that is, (ii) holds.
\end{proof}

\section{Torsors: reduction of structure to parabolic subgroups}
\label{sect.reduction-to-parabolic}

\begin{defin}\label{defin:reduction-of-structure}
    Let $k$ be a field, let $\varphi\colon H\to G$ be a homomorphism of $k$-groups, and let $K/k$ be a field extension. We say that a $G$-torsor $E\to \on{Spec}(K)$ \emph{admits reduction of structure to $H$} if one of the following equivalent conditions is satisfied:
    \begin{enumerate}
        \item the $G$-torsor $E$ is induced by some $H$-torsor over $K$ via the map $\varphi$;
        \item the class of $E$ in $H^1(K,G)$ is in the image of the map $\varphi_*\colon H^1(K,H)\to H^1(K,G)$.
        \item the morphism $\on{Spec}(K)\to B_kG$ which classifies $E$ admits a factorization of the form $\on{Spec}(K)\to B_kH\xrightarrow{B_k\varphi} B_kG$.
    \end{enumerate} 
    This defines a property of $G$-torsors over fields (cf. \Cref{defin:property-torsors}(1)). 
\end{defin}

Let $k$ be a field, let $G$ be a $k$-group, and let $H\to G$ be a homomorphism of $k$-groups. Let $S$ be a $k$-scheme of finite type, let $X$ be an $S$-scheme of finite type with geometrically reduced $S$-fibers, and let $\alpha\colon E \to X$ be a $G$-torsor. For every geometric point $s$ of $S$, one can ask whether the
$G$-torsor $\alpha_s\colon E_s \to X_s$ admits reduction of structure to $H$ generically. What is the locus of (images of) geometric points $s$ of $S$ such that $\alpha_s \colon E_s \to X_s$ has this property? In this section we will consider this question in a special case, where the answer can be deduced directly from \Cref{thm.rat-sections}.  

\begin{lemma}\label{lem-g-p-proper}
    Let $k$ be a field, let $G$ be a $k$-group such that $G^0$ is reductive, let $P$ be a $k$-subgroup of $G$ such that $P^0$ is a parabolic subgroup of $G^0$, let $V$ be a $k$-scheme of finite type, and let $E\to V$ be a $G$-torsor. Then the quotient $E/P$ is represented by a proper $V$-scheme.
\end{lemma}

\begin{proof}
Suppose first that $P$ is connected. In this case, we will prove that $E/P$ is represented by a projective $V$-scheme. Then $P=P_G(\lambda)$ is the parabolic subgroup associated to some cocharacter $\lambda\colon \mathbb{G}_m\to G$; see \cite[p. 49 and Proposition 2.2.9]{conrad-gabber-prasad}. Fix a $k$-subgroup embedding $G\subset \on{GL}_n$ for some $n\geqslant 1$, and let $P_{\on{GL}_n}(\lambda)$ be the parabolic subgroup of $\on{GL}_n$ associated to $\lambda$, this time viewed as a cocharacter of $\on{GL}_n$. By \cite[Proposition 2.1.8(3)]{conrad-gabber-prasad}, we have $P_{\on{GL}_n}(\lambda) \cap G=P$. It follows that the induced proper morphism $G/P\to \on{GL}_n/P_{\on{GL}_n}(\lambda)$ is a monomorphism and hence a closed immersion; see \cite[Tag 04XV]{stacks-project}. We obtain a closed immersion $E/P\hookrightarrow E'/P_{\on{GL}_n}(\lambda)$, where $E'\to V$ is the $\on{GL}_{n}$-torsor induced by $E \to V$. We may thus assume that $G=\on{GL}_n$. In this case, since every $\on{GL}_n$-torsor is Zariski-locally trivial, by passing to a Zariski open cover of $S$ we may suppose that $E \to V$ is trivial, so that $E/P\cong (\on{GL}_n/P_{\on{GL}_n}(\lambda))\times_kV$ is a projective $V$-scheme. This proves that $E/P$ is a projective $V$-scheme when $P$ is connected.

Now let $P$ be arbitrary, and consider the finite \'etale $k$-group $\Gamma\coloneqq P/P^0$. Since properness is Zariski local on the target, in order to prove that $E/P$ is representable by a proper $V$-scheme we may assume that $V$ is affine. By the first part of the proof, $E/P^0$ is projective over $V$, and hence it is quasi-projective over $k$. We deduce that $E/P\cong (E/P^0)/\Gamma$ is representable by a (quasi-projective) scheme over $k$. The morphism $E/P\to V$ is proper because so is $E/P^0\to V$.
\end{proof}

\begin{prop} \label{prop.parabolic} Let $k$ be a field, let $G$ be a $k$-group whose identity component $G^0$ is reductive, let $P\subset G$ be a $k$-subgroup such that $P^0$ is a parabolic subgroup of $G^0$, let $S$ be a $k$-scheme of finite type, and let $\alpha\colon E \to X$ be a $G$-torsor of $S$-schemes of finite type. Assume that $X$ is $S$-flat and with geometrically reduced $S$-fibers. Then there exists a countable collection 
of closed subschemes $\{S_j\}_{j\geqslant 1}$ of $S$ such that for every geometric point  
$s$ of $S$ the following are equivalent:
\begin{itemize}
    \item[(i)] the $G$-torsor $\alpha_s\colon E_s \to X_s$ admits reduction of structure to $P$ generically;
    \item[(ii)] $s$ lies in $S_j$ for some $j \geqslant 1$.
\end{itemize}
\end{prop}

\begin{proof} 
By \Cref{lem-g-p-proper}, the map $f\colon E/P\to X$ induced by $\alpha$ is a proper morphism of schemes. Moreover, by \cite[Proposition I.5.4.37]{serre1997galois}, for every
geometric point $s$ of $S$, the $G$-torsor $\alpha_s \colon E_s\to X_s$ admits reduction of structure to $H$ generically if and only if the proper morphism $f_s \colon (E/P)_s = E_s/P_s \to X_s$ admits a rational section. The conclusion now follows from~\Cref{thm.rat-sections}.
\end{proof}

\begin{rmk} \label{rem.parabolic}
Example~\ref{ex.quadric} with $m = 0$ is a special case of Proposition~\ref{prop.parabolic}. Indeed, a quadric bundle $Q \to Y$ of rank $n$ is the same thing as a $\operatorname{O}_n$-torsor $E \to Y$, 
where $\operatorname{O}_n$ is the group of isometries of a split $n$-dimensional quadratic form 
over $k$. A fiber $Q_s \to Y_s$ over a geometric point $s \colon \Spec(k) \to S$ has generic Witt index $\geqslant d$ if and only if the $\operatorname{O}_n$-torsor $E_s \to Y_s$ admits reduction of structure generically to a parabolic subgroup $P_d$ of $\operatorname{O}_n$, where $P_d$ is the stabilizer of an $d$-dimensional isotropic subspace for the split quadratic form in a $n$-dimensional vector space. The isotropic Grassmannian 
$\Gr^{\rm iso}(d, Q)$ is isomorphic to $E/P_d$ over $Y$ or equivalently, the homogeneous space $\operatorname{O}_n/P_d$ twisted by $E$.

Similarly, Example~\ref{ex.azumaya} follows from
Proposition~\ref{prop.parabolic} with $G = \PGL_n$. Indeed,
an Azumaya algebras $\mc{A}$ of degree $n$ over $Y$ is the same thing as a $\PGL_n$ torsor $E \to Y$. A fiber $\mathcal{A}_s$ over a geometric point $s$ of $S$ has generic index dividing $d$ if and only if $E_s \to Y_s$ admits reduction of structure to a suitable parabolic subgroup $P^{(d)}$ of $\PGL_n$ generically. Once again, the Severi-Brauer scheme $\SB(d, \mathcal{A})$ can be identified with $E/P^{(d)}$ over $Y$.   
\end{rmk}

\section{A geometric variant of the Hilbert Irreducibility Theorem}
\label{sect.hilbert-irreducibility} 

Let $k$ be an algebraically closed field and $K$ be a finitely generated field extension of $k$. Consider a polynomial 
\begin{equation} \label{e.hilbert-irreducibility}
f_t(x) = x^n + f_1(t_1, \ldots, t_m) x^{n-1} + \ldots + f_n(t_1, \ldots, t_m), 
\end{equation}
where $f_1, \ldots, f_n \in K[t_1, \ldots, t_m]$.
Let $\Delta(t_1, \ldots, t_m) \in K[t_1, \ldots, t_m]$ be the discriminant of
$f_t(x)$ (viewed as a polynomial in $x$), and $U$
be the open subvariety of $\mathbb A^n$, where $\Delta(t_1, \ldots, t_m) \neq 0$. Assume that $U \neq \emptyset$ or equivalently, $U$ is dense in $\mathbb A^m$.
For every $k$-point $t = (t_1, \ldots, t_n) \in U(k)$, we can ask whether the polynomial $f_t(x)$ is irreducible over $K$.

Note that $f_t(x)$ has no multiple roots for any $t \in U$.
Hence, for every $t \in U$, it makes sense to talk about the Galois group of $f_t(x)$ over $K$. This group, which we will denote by $\Gal(f_t)$, faithfully acts on the roots of $f_t(x)$; an ordering of the roots of $f_t(x)$ determines an embedding of $\Gal(f_t)$ in $\Sym_n$. A different ordering of the roots gives rise to an equivalent embedding, i.e., the same embedding followed by an inner automorphism of $\Sym_n$. 
Note that $f_t(x)$ is irreducible if and only if the natural action of $\Gal(f_t) \subset \Sym_n$ on $\{ 1, \ldots, n \}$ is transitive.

\smallskip
\begin{prop} \label{prop.galois-group} Let $H$ be a finite subgroup of $\Sym_n$.
Then there exists a countable collection of closed subvarieties $\{S_j\}_{j\geqslant 1}$ of $\mathbb A^m$ such that a geometric point $t = (t_1, \ldots, t_m) \in U(k)$  lies in $S_j$ for some $j \geqslant 1$ if and only if $\Gal(f_t)$ is conjugate to a subgroup of $H$.
\end{prop}

\begin{proof}
Let $K$ be the function field of an irreducible affine variety $Z$. After replacing $Z$ by a suitable affine subvariety, we may assume that the coefficients $f_i(t_1, \ldots, t_m)$ are regular functions on $Z$ for every $i = 1, \ldots, n$. Let $S = U$, $Y = Z \times U$, and $X \subset Y \times \mathbb A^1$ be the zero locus of $f_t(x)$. Here $t = (t_1, \ldots, t_m)$ ranges over $U$ and $x$ ranges over $\mathbb A^1$.
By the definition of $U$, the natural projection $\pi \colon X \to Y$ is an \'etale morphism of degree $n$. Hence, it is represented by a class in $H^1(Y, \Sym_n)$ or equivalently, by a $\Sym_n$-torsor $\tau \colon E \to Y$. Moreover, for a geometric point $t \colon \Spec(k) \to S = U$, the Galois group $\Gal(f_t)$ is contained in $H$ (up to conjugacy in $\Sym_n$) if and only if the $\Sym_n$-torsor $E_t \to Y_t$ admits reduction of structure to $H$ over the generic point of $Y_t \simeq Z$.
The existence of $\{S_j\}_{j\geqslant 1}$ now follows from
\Cref{prop.parabolic}, with $G = \Sym_n$. Note that since $\Sym_n$ is a finite group, every subgroup $H$ is parabolic.
\end{proof}

\begin{prop} \label{prop.irreducible} There is a countable collection of closed subvarieties $S_j$ of $\mathbb{A}^m$
such that a geometric point $t = (t_1, \ldots, t_m) \in \mathbb{A}^m(k)$ 
lies in one of these subvarieties if and only if $f_t(t)$ is reducible over $K$.
\end{prop}

\begin{proof} 
Note that $f_t(x)$ is reducible over $K$ in one of two cases: 

\smallskip
(i) $f_t(x)$ has a multiple root, i.e., $t$ lies in the discriminant locus $S_0 = \mathbb A^m \setminus U$ given by $\Delta(t) = 0$, or

\smallskip
(ii) $t$ lies in $U$ and the Galois group $\Gal(f_t)$  is conjugate to a subgroup of
$H_i = \Sym_i \times \Sym_{n-i} \subset \Sym_n$ for some $i = 1, \ldots, \lfloor n/2\rfloor$. More specifically, the Galois group of $f_t(x)$ is contained 
in $H_i$ (up to conjugacy) if and only if $f_t(x)$ decomposes as a product of polynomials of degree $i$ and $n-i$ with coefficients in $K$. By~\Cref{prop.galois-group} there is a countable collection of closed subschemes $S_{H_i}[j] \subset U$, where $j$ ranges over the non-negative integers with the following property: for $t \in U(k)$, the Galois group of $f_t(x)$ is contained in $H_i$ if and only if $t$ lies in $S_{H_i}[j]$ for some $j \geqslant 1$. Let $S_i[j]$ be the 
Zariski closure of $S_{H_i}[j]$ in $\mathbb A^m$.

\smallskip
In summary, the countable collection of closed subvarieties of $\mathbb A^m$ consisting of $S_1 = \mathbb{A}^m \setminus U$ and $S_i[j]$ ($i = 1, \ldots, \lfloor n/2\rfloor$ and $j = 1, 2 \ldots$) has the required properties.
\end{proof}

\section{Torsors: reduction of structure to general subgroups}
\label{sect.reduction-to-general}

Let $k$ be a field, and let $H \to G$ be a $k$-group homomorphism. 
In this section we investigate how the property of admitting reduction of structure to $H$ behaves in a family of $G$-torsors $E \to X$ parametrized by a scheme $S$. In~\Cref{prop.parabolic} we showed that if $G^0$ is reductive and $H$ is a $k$-subgroup of $G$ such that $H^0$ is a parabolic subgroup of $G^0$, then there there exists a countable union of closed subschemes $S_j$ ($j = 1, 2, \ldots$) of $S$ such that the fiber $E_s \to X_s$ over a geometric point $s \in S$ admits reduction of structure to $H$ if and only if $s$ lies in $S_j$ for some $j \geqslant 1$.  If we do not assume that $H$ contains a parabolic subgroup of $G$, then $X_s/H$ may not be proper over $Y$, and hence \Cref{thm.rat-sections} cannot be applied to the morphism $X/H\to Y$.  Nevertheless, in this section we will show that \Cref{prop.parabolic} remains valid in this more general setting, under a mild assumption on the characteristic of $k$.

\begin{defin} \label{def.good-characteristic}
Let $k$ be an algebraically closed field, and let $G$ be an affine $k$-group. We say that $G$ is in \emph{good characteristic} if there exists a finite discrete subgroup $\Lambda \subset G$, of order invertible in $k$, such that for every field extension $K/k$ the map $H^1(K, \Lambda)\to H^1(K,G)$ is surjective. 
\end{defin}

Note that by~\cite[Theorem 1.1]{cgr}, if $\on{char}(k)=0$, then every affine $k$-group is in good characteristic. We are now ready to state the main result of this section.

\begin{thm}\label{thm.reduction-of-structure} Let $k$ be a field, let $G$ and $H$ be smooth affine $k$-groups, let $H\to G$ be a $k$-group homomorphism, let $S$ be a $k$-scheme of finite type, let $X\to S$ be a flat morphism of finite type with geometrically integral fibers, and let $\alpha\colon E \to X$ be a $G$-torsor. Assume that $H$ is in good characteristic. Then there exists a countable collection of closed subschemes $\{S_j\}_{j\geqslant 1}$ of $S$ 
such that for every geometric point $s$ of $S$ the following are equivalent:
\begin{itemize}
    \item[(i)] the $G$-torsor $\alpha_s\colon E_s \to X_s$ admits reduction of structure to $H$ generically;
    \item[(ii)] $s$ lies in $S_j$ for some $j \geqslant 1$.
\end{itemize}
\end{thm}

Note that here we assume only that $H$ is in good characteristic. There is no characteristic assumption on $G$.
Our proof of \Cref{thm.reduction-of-structure} will rely on the following.

\begin{lemma}\label{rigidity.reduction-of-structure}
    Let $k$ be an algebraically closed field, let $K/k$ and $k'/k$ be field extensions, and let $K'$ be the fraction field of the domain $K\otimes_kk'$. Let $H\to G$ be a $k$-group homomorphism, and let $E\to \on{Spec}(K)$ be a $G$-torsor. Then $E$ admits reduction of structure to $H$ if and only if $E_{K'}\to \on{Spec}(K')$ admits reduction of structure to $H_{K'}$.
\end{lemma}

    Recall that, since $k$ is algebraically closed, if $A$ and $B$ are $k$-algebras and domains, the $k$-algebra $A\otimes_kB$ is also a domain. In particular, $K\otimes_kk'$ is a domain. 

\begin{proof}
If the $G$-torsor $E\to \on{Spec}(K)$ admits reduction of structure to $H$, then clearly so does $E_{K'}\to \on{Spec}(K')$.

Conversely, suppose that $E_{K'}\to \on{Spec}(K')$ admits reduction of structure to $H$. We may assume without loss of generality that the extensions $K/k$ and $k'/k$ are finitely generated. Then there exist integral $k$-varieties $V$ and $W$ such that $k(V)=K$ and $k(W)=k'$, respectively and consequently, $K'=k'(V)=k(V\times_kW)$. Shrinking $V$ if necessary, we may assume that there exists a $G$-torsor $\mc{E}\to V$ whose generic fiber is $E\to \on{Spec}(K)$. Since $E_{K'}\to \on{Spec}(K')$ admits reduction of structure to $H$, there exists a dense open subscheme $U\subset V\times_kW$ such that the pullback $\mc{E}|_U$ along the first projection $U\to V$ admits reduction of structure to $H$. Since $k$ is algebraically closed and $U$ is non-empty, there exists a $k$-point $u=(v,w)\in U(k)$. Then $U\cap (V\times_k \{w\})=V'\times_k\{w\}$ for some dense open subscheme $V'\subset V$. Restricting $\mc{E}|_U$ to $V'\times_k\{w\}$, we conclude that $\mc{E}|_{V'}\to V'$ admits reduction of structure to $H$, and hence so does the generic fiber $E\to \on{Spec}(K)$, as desired.
\end{proof}

\begin{lemma}\label{specialization.reduction-of-structure}
    Let $k$ be a field, let $R$ be a complete discrete valuation ring containing $k$ with fraction field $K$ and residue field $\kappa$. Let $H\to G$ be a homomorphism of smooth affine $k$-groups and let $E\to \on{Spec}(R)$ be a $G$-torsor. Assume that the $k$-unipotent radical $U$ of $G^0$ is split over $k$, that the quotient $G^0/U$ is reductive, and that  $H$ is in good characteristic. If $E_K$ admits reduction of structure to $H$, then so does $E_\kappa$.
\end{lemma}

Recall that a smooth connected unipotent $k$-group is $k$-\emph{split} if is an iterated extension of additive groups $\mathbb{G}_{a,k}$.

\smallskip

When
$\Char(k) = 0$, $H \to G$ can be an arbitrary homomorphism of affine $k$-groups. All other assumptions 
on $H$ and $G$ in \Cref{specialization.reduction-of-structure}
are automatic in characteristic $0$. Note also that our proof is considerably simpler in this setting; see below.

\begin{proof}
    First suppose that $\on{char}(k)=0$. By Cohen's structure theorem \cite[Tag 032A]{stacks-project}, there is an isomorphism of $k$-algebras $R\cong \kappa[\![t]\!]$. Replacing $k$ by $\kappa$, we may assume that $R = k[\![t]\!]$, so that $K=k(\!(t)\!)$, and we let $K_\infty\coloneqq \cup_{n\geqslant 1}k(\!(t^{1/n})\!)$ be the field of Puiseux series over $k$, which is an algebraic closure of $K$. 
    By \cite[Theorem 5.1]{florence2006points}, the pullback maps $H^1(k,G)\to H^1(K,G)$ and $H^1(k,H)\to H^1(K,H)$ are bijective. By \cite[Expos\'e XIV, Proposition 8.1]{SGA3III} the reduction maps $H^1(R,G)\to H^1(k,G)$ and $H^1(R,H)\to H^1(k,H)$ are bijective. Therefore the maps $H^1(R,G)\to H^1(K,G)$ and $H^1(R,H)\to H^1(K,H)$ are bijective. This tells us that if $E_K$ admits reduction of structure to $H$, so does $E$. Reducing modulo $t$, we conclude that $E_k$ admits reduction of structure to $H$, as desired.

Now assume that $\on{char}(k) = p >0$. Recall that we are assuming that $H$ is in good characteristic. This means that there exists a finite discrete group $\Lambda\subset H$ of order prime to $p$ such that for every field extension $K/k$ the map $H^1(K,\Lambda)\to H^1(K,H)$ is surjective; see \Cref{def.good-characteristic}. Replacing $H$ by $\Lambda$, we may thus assume that $H$ is a finite discrete group of prime-to-$p$ order. Let $\pi\in R$ be a uniformizer, and let $K_\infty\coloneqq \cup_n K(\pi^{1/n})\subset K_{\on{sep}}$, where $n$ ranges over all positive prime-to-$p$ integers, and where the $\pi^{1/n}\in K_{\on{sep}}$ are chosen so that for all positive prime-to-$p$ integers $m,n$ we have $(\pi^{1/mn})^m=\pi^{1/n}$. We have the following commutative diagram
    \[
\xymatrix{
H^1(\kappa,H) \ar[d] & H^1(R,H) \ar[l] \ar[r] \ar[d] & H^1(K,H) \ar[r] \ar[d] & H^1(K_\infty,H) \ar[d] \\
H^1(\kappa,G) & H^1(R,G) \ar[l] \ar[r] & H^1(K,G) \ar[r] & H^1(K_\infty,G)
}
\]
    Let $\alpha\in H^1(R,G)$ be the class of $E$. By assumption, $\alpha_K$ comes from $H^1(K,H)$, and so $\alpha_{K_\infty}$ comes from $H^1(K_\infty,H)$. Let $K\subset K_{\on{un}}\subset K_{\on{sep}}$ be the maximal unramified extension. Then $K_{\on{un}}/K$ is Galois, $K_{\on{un}}\cap K_{\infty}=K$, and the absolute Galois group of the compositum $K_{\on{un}}K_\infty$ is a pro-$p$ group. Because $H$ is discrete and of finite prime-to-$p$ order, the composite $H^1(R,H)\to H^1(K_\infty,H)$ is identified with the map
\[
\resizebox{\textwidth}{!}{$
\on{Hom}(\on{Gal}(K_{\on{un}}/K),H)/\mkern-2mu\raisebox{-0.4ex}{\(\sim\)} 
\xrightarrow{\;\simeq\;} 
\on{Hom}(\on{Gal}(K_{\on{un}}K_{\infty}/K_{\infty}),H)/\mkern-2mu\raisebox{-0.4ex}{\(\sim\)} 
\xrightarrow{\;\simeq\;} 
\on{Hom}(\on{Gal}(K_{\on{sep}}/K_{\infty}),H)/\mkern-2mu\raisebox{-0.4ex}{\(\sim\)}
$}
\]
    where $\sim$ denotes the equivalence relation induced by $H$-conjugation. We deduce that the map $H^1(R,H)\to H^1(K_\infty,H)$ is bijective, and hence in particular there exists $\beta\in H^1(R,H)$ such that $\beta_{K_\infty}$ maps to $\alpha_{K_\infty}$ in $H^1(K_\infty,G)$. Because $K_\infty$ is the filtered colimit of the subfields $K(\pi^{1/n})$, we deduce that $\beta_{K(\pi^{1/n})}$ maps to $\alpha_{K(\pi^{1/n})}$ in $H^1(K(\pi^{1/n}),G)$ for some $n\geqslant 1$. The integral closure $R'$ of $R$ in $K(\pi^{1/n})$ is a discrete valuation ring by \cite[Tag 09E8]{stacks-project}), and it has residue field $\kappa$ because the extension $K(\pi^{1/n})/K$ is totally ramified. Replacing $R$ by $R'$, we may thus assume that $n=1$, that is, that $\beta_K$ maps to $\alpha_K$ in $H^1(K,G)$. Assuming that the map $H^1(R,G)\to H^1(K,G)$ is injective, we conclude that $\beta$ maps to $\alpha$ in $H^1(R,G)$, and hence that $\beta_\kappa$ maps to $\alpha_\kappa$ in $H^1(\kappa,G)$, as desired.
    
    It remains to prove that the map $H^1(R,G)\to H^1(K,G)$ is injective. When $G^0$ is reductive, this is \cite[Lemma 3.5(b)]{reichstein2022behavior}. For the general case, let $U\subset G^0$ be the $k$-unipotent radical of $G^0$. By assumption, $U_{\cl{k}}$ is the unipotent radical of $(G^0)_{\cl{k}}$, and hence the quotient $G^0/U$ is reductive. As $U$ is a characteristic $k$-subgroup of $G^0$, it is a normal $k$-subgroup of $G$. Letting $\cl{G}\coloneqq G/U$, the connected component of $\cl{G}$ is reductive, and hence the map $H^1(R,\cl{G})\to H^1(K,\cl{G})$ is injective. By a twisting argument, the injectivity of $H^1(R,G)\to H^1(K,G)$ will follow once we prove that $H^1(R,\prescript{}{\tau}{U})$ is trivial for all $G$-torsors $\tau$ over $R$, where $\prescript{}{\tau}{U}$ is the twist of $U$ by $\tau$ via the conjugation $G$-action. Since $U$ is smooth over $k$, the twist $\prescript{}{\tau}{U}$ is smooth over $R$, so that by \cite[Expos\'e XIV, Proposition 8.1]{SGA3III} the reduction map $H^1(R,\prescript{}{\tau}{U})\to H^1(\kappa,\prescript{}{\tau}{U})$ is bijective. As $G$ is smooth over $k$, the $G$-torsor $\tau$ is \'etale-locally trivial, and hence the special fiber $(\prescript{}{\tau}{U})_\kappa$ is a smooth connected unipotent $\kappa$-group which becomes split after a finite separable field extension of $\kappa$. By \cite[Theorem B.3.4]{conrad-gabber-prasad}, this implies that $(\prescript{}{\tau}{U})_\kappa$ is $\kappa$-split, and hence that $H^1(\kappa,\prescript{}{\tau}{U})$ is trivial. Thus $H^1(R,\prescript{}{\tau}{U})$ is trivial, as desired.
\end{proof}

\begin{proof}[Proof of \Cref{thm.reduction-of-structure}]
    Let $\mc{P}$ be the property of $G$-torsors over fields given by admitting reduction of structure to $H$. The property $\tilde{\mc{P}}$ is the property of admitting reduction of structure to $H$ generically. Property $\mc{P}$ satisfies the rigidity condition by \Cref{rigidity.reduction-of-structure}. In view of \Cref{thm.main.torsors}, it suffices to prove that $\mc{P}$ satisfies the specialization condition with respect to $(S,\alpha)$. By \Cref{reduce-to-dvr-torsors}, we are reduced to the following situation. Let $A$ be a complete discrete valuation ring containing $k$, and let $T\to \on{Spec}(A)$ be a $G$-torsor. We assume that a geometric generic fiber of the $G$-torsor $T\to \on{Spec}(A)$ admits reduction of structure to $H$, and we must show that a geometric special fiber of $T\to \on{Spec}(A)$ also admits reduction of structure to $H$. Since $k$ is algebraically closed, the unipotent radical $U$ of $G$ is $k$-split (see \cite[Expos\'e XVII, Corollaire 4.1.3]{SGA3III}) and $G^0/U$ is reductive (see \cite[Proposition 1.1.9(2)]{conrad-gabber-prasad}). The conclusion follows from \Cref{specialization.reduction-of-structure}.
\end{proof}
\begin{example} \label{ex.split} 
Let $k$ be an algebraically closed field, let $G$ be a $k$-group, let $K/k$ be a field extension, and let $\tau$ be a $G$-torsor over $K$. The following are equivalent:
\begin{enumerate}
    \item $\tau$ is split,
    \item $\tau$ admits reduction of structure to the trivial group,
    \item $\ed_k(\tau)=0$, 
    \item $\ed_k(\tau)\leqslant 0$,
    \item $\rd_k(\tau)=0$,
    \item $\rd_k(\tau)\leqslant 0$.
\end{enumerate}
The equivalences (1) $\Longleftrightarrow$ (2) $\Longleftrightarrow$ (3) $\Longleftrightarrow$ (4) and (5) $\Longleftrightarrow$ (6) are by definition; the equivalence of (1) and (4) is proved in \cite[Lemma 7.8]{reichstein-hp13}. Therefore \Cref{thm.reduction-of-structure} implies the following: for a $k$-scheme of finite type $S$, a $G$-torsor $E \to X$ over $S$ such that $X \to S$ is flat, of finite type, and has geometrically integral fibers, there exists a countable collection of closed subschemes $S_j\subset S$ such that for every a geometric point $s$ of $S$, the $G$-torsor $E_{k(Y_s)} \to \on{Spec}(k(X_s))$ satisfies the equivalent properties (1)-(6) if and only if $s$ lies in $S_j$ for some $j\geqslant 1$.
\end{example}

\section{Symbol length for central simple algebras}
\label{sect.symbol-length}

Let $A$ be a central simple algebra of index $n$ and exponent dividing $m$ over a field $K$. Assume that $K$ contains a primitive $m$th root of unity $\zeta$.  By the Merkurjev-Suslin Theorem, $A$ is Brauer equivalent to a tensor product of symbol algebras 
\begin{equation} \label{e.merkurjev-suslin} A \sim (a_1, s_1)_m \otimes \ldots \otimes (a_l, b_l)_m , 
\end{equation}
for some $l\geqslant 0$ and $a_1, s_1, \ldots, a_l, b_l \in K^*$. Here by the symbol algebra $(a, b)_m$, we mean
the $m^2$-dimensional associative $K$-algebra generated by two elements, $x$ and $y$, subject to relations $x^m = a$,
$y^m = b$ and $yx = \zeta xy$.
The decomposition~\eqref{e.merkurjev-suslin} is not unique. We will denote the minimal number $l$ of symbol algebras by $l(A, m)$. In other words, $l(A, m)$ is the symbol length of the class of $A$ in $H^2(K, \mu_m)$. Some authors refer to $l(A, m)$ as {\em the Merkurjev-Suslin number} of $A$. This number is independent of the choice of $\zeta$ in $K$. In other words, $l(A, m)$ makes sense whenever $K$ contains a primitive $m$th root of unity, we do not need to choose a specific one.
Note also that if the exponent $e$ of $A$ is strictly less than the  index $n$, then we have some freedom in choosing $m$, and $l(A, m)$ may depend on this choice. For example $l(A, e)$ may be strictly larger than $l(A, n)$. For a more detailed discussion of this invariant, we refer the reader to \cite{tignol83}.

\begin{prop} \label{prop.symbol-length}
Let $m$ be a positive integer, let $k$ be a field containing a primitive $m$th root of unity, let $S$ be a $k$-scheme of finite type, let $X\to S$ be a flat morphism, and let $\mathcal{A}$ be an Azumaya algebra over $X$ of exponent dividing $m$. Given a geometric point $s$ of $S$, we will say that $l(\mathcal{A}_s, m) \leqslant d$ if for every generic point $\Spec(K) \to X_s$ of an irreducible component of $X_s$, the Merkurjev-Suslin number $l((\mathcal{A}_s)_K, m)$ of the induced central simple algebra $(\mathcal{A}_s)_K$ is $\leqslant d$.

Then for any integer $d \geqslant 0$ there exists a countable collection $\{S_j\}_{j\geqslant 1}$ of closed subschemes of $S$
such that for a geometric point $s$ of $S$, we have
$l(\mathcal{A}_s, m) \leqslant d$  if and 
only if $s$ lies in $S_j$ for some $j \geqslant 1$.
\end{prop}

Our proof of this proposition will rely on the following lemma. Let $\Gamma$ be a finite abelian group of order $n$ and exponent dividing $m$. Let $K$ be a field containing a primitive $m$th root of unity
 and $V = K[\Gamma]$ be the group algebra. Consider the regular representation 
$\rho \colon \Gamma \hookrightarrow \GL(V) = \GL_n$ given by $\rho(a) \colon  g  \mapsto a \cdot g$ for every $a \in \Gamma$ and every $g \in \Gamma \subset V$ and the representation $\rho^* \colon \Gamma^* \to \GL(V) = \GL_n$ of the dual group $\Gamma^*:= \operatorname{Hom}(\Gamma, \mathbb{G}_m)$, given by $\chi \colon g \mapsto \chi(g) g$ for every character $\chi \colon \Gamma \to \mathbb G_m$ and every $g \in \Gamma \subset V$. It is easy to see that \[ \rho^*(\chi) \rho(a) = \chi(a) \cdot \rho(a) \rho^*(\chi) . \]
This shows that $\rho$ and $\rho^*$ do not define a representation of $\Gamma \times \Gamma^* \to \GL(V) = \GL_n$. On the other hand, since $\chi(a)$ is an $m$th root of unity, they do define an embedding 
\begin{equation} \label{e.rho_Gamma}
\rho_{\Gamma} \colon= \rho \times \rho^* \, \colon \Gamma \times \Gamma^* \hookrightarrow \GL_n/\mu_m.
\end{equation}
We will need the following slight generalization of this construction. For any integer $r \geqslant 1$, note that there is a natural isomorphism $(\Gamma^r)^* \simeq (\Gamma^*)^r$. This allows us to define the diagonal embedding $\Delta_{\Gamma, r} \colon \Gamma \times \Gamma^* \hookrightarrow \Gamma^r \times (\Gamma^r)^*$. Let 
\[ \rho_{\Gamma, r} \colon \Gamma \times \Gamma^* \hookrightarrow \GL(V^r)/\mu_m = \GL_{nr}/ \mu_m \]
be the composition of
$\Delta_{\Gamma, r}$ and $\rho_{\Gamma^r}$. Note that for $r = 1$, we have $\Delta_{\Gamma, 1} = \operatorname{id}$, and $\rho_{\Gamma, 1} = \rho_{\Gamma}$.

\begin{lemma} \label{lem.symbol} 
Suppose $A$ is a central simple algebra of degree $n$ and exponent dividing $m$ over a field $K$. Assume that $K$ contains a primitive $m$th root of unity, and $n = m^d r$ for some integers $d \geqslant 0$ and $r \geqslant 1$. Then the following conditions are equivalent:

\smallskip
(a) $A$ is $K$-isomorphic to a tensor product $(a_1, b_1)_m \otimes_K \ldots \otimes_K (a_d, b_d)_m \otimes_K \Mat_r(K)$ for some $a_1, b_1, \ldots, a_d, b_d \in K^*$.

\smallskip
(b) the class of $A$ in $H^1(K, \GL_{n}/\mu_m)$ lies in the image of the natural map
\[ (\rho_{\Gamma, r})_* \colon H^1(K, \Gamma \times {\Gamma}^*) \longrightarrow H^1(K, \GL_{nr}/\mu_m), \]
where $\Gamma = (\mathbb{Z}/m \mathbb Z)^d$.
\end{lemma}

\begin{proof} We proceed in two steps.

\medskip
{\bf Step 1.} The special case, where $r = 1$ is well known; see, e.g., Proposition 4.7.3 and its proof in~\cite{gille2017central}.

\medskip
{\bf Step 2.} We now turn to the general case, where $r$ is an arbitrary positive integer. Consider the diagonal embedding $\GL_n \to \GL_{nr}$ given by
\[ g \mapsto \begin{pmatrix} g & 0 & 0 & \ldots & 0 \\
                             0 & g & 0 & \ldots & 0 \\
                             \vdots& \vdots & \vdots & \vdots & \vdots            \\
                             0 & 0 & 0 & \ldots & g 
\end{pmatrix}. \]
This map induces an embedding $\Delta_{n, r} \colon \GL_n/\mu_m \to \GL_{nr}/\mu_m$ for any $m$ dividing $n$. Now observe that the following diagram commutes:
\[  \xymatrix{  \GL_n/\mu_m \ar@{->}[r]_{\Delta_{n, r}}  &   \GL_{nr}/\mu_m   \\ 
\Gamma \times \Gamma^* \ar@{->}[u]_{\rho_{\Gamma}} \ar@{->}[r]_{\Delta_{\Gamma, r}}  &   \Gamma^r \times (\Gamma^r)^* \ar@{->}[u]_{\rho_{\Gamma^r}}    } \]
Condition (a) holds if and only if the following two conditions hold:

\smallskip
(a1) $A$ is of the form $\Mat_s(A_0)$ for some central simple $K$-algebra $A_0$ of degree $m^d$. Note that if the $A_0$ exists, it is uniquely determined by $A$, up to $K$-isomorphism.

\smallskip
(a2) Moreover, $A_0$ is $K$-isomorphic to a tensor product $(a_1, b_1)_m \otimes \ldots (a_d, b_d)_m$ of $d$ symbol algebras for some $a_1, b_1, \ldots, a_d, b_d \in K^*$.

\smallskip
Now observe that (a1) is equivalent to saying that class $[A] \in H^1(K, \GL_{nr}/\mu_m)$ of $A$   
is the image of $[A_0]$ under the induced morphism \[(\Delta_{n, r})_* \colon H^1(K, \GL_n/\mu_m) \to H^1(K, \GL_{nr}/\mu_m), \]
for some central simple $K$-algebra $A_0$ of degree $m^d$ and exponent dividing $m$.
On the other hand, by Step 1, (a2) is equivalent to saying that the class of $[A_0]$ descends to $\Gamma \times \Gamma^*$. 

In summary, (a) holds $\Longleftrightarrow$ both (a1) and (a2) hold $\Longleftrightarrow$ $[A] \in H^1(K, \GL_{nr}/\mu_m)$ descends to $\Gamma \times \Gamma^*$ $\Longleftrightarrow$ (b) holds.
\end{proof}

\begin{proof}[Proof of \Cref{prop.symbol-length}]
Let $E \to X$ be the $\GL_n/\mu_m$-torsor whose image in $H^2(X, \mu_n)$ is the Brauer class of the Azumaya algebra $\mathcal{A}$ from the statement of \Cref{prop.symbol-length}. After replacing $\mathcal{A}$ by a Brauer equivalent Azumaya algebra over $X$, we may assume that $E \to X$ represents $\mathcal{A}$. Fix an integer $l \geqslant 0$. 

Now let $s$ be a geometric point of $S$ and $\mathcal{A}_s$ be the fiber of $\mathcal{A}$ over $s$. We would like to know whether or not the symbol length of $\mathcal{A}_s$ over the generic point of $X_s$ (or more generally, over the generic point of every irreducible component of $X_s$) is $\leqslant l$. We now use \Cref{lem.symbol} to restate this property of $s$ in terms of the $\GL_n/\mu_m$-torsor $E \to X$.

In general, let $m, n$ be positive integers such that $m$ divides $n$. Choose relatively prime positive integers $r$ and $t$ such that $r/t = n/m^d$. The positive integers $n, m, r$ and $t$ will be fixed from now on. By Wedderburn's Theorem, a central simple algebra $A$ over a field $K$, of degree $n$ and exponent dividing $m$, is Brauer-equivalent to $(a_1, b_1)_m \otimes \ldots \otimes (a_d, b_d)_m$ if and only if 
\[ \text{$A \otimes \Mat_t(K)$ is $K$-isomorphic to $(a_1, b_1)_m \otimes \ldots \otimes (a_d, b_d)_m \otimes \Mat_r(K)$.} \]
By \Cref{lem.symbol} this translates into:

\begin{align}
 l(A, m) \leqslant d 
\Longleftrightarrow\ 
& \text{the class of } \Mat_t(A) \text{ in } H^1(K, \GL_{nr}/\mu_m) \text{ admits reduction }\label{eq:reduction} \\
&  
\text{ of structure to } \Gamma \times \Gamma^* \text{ via } 
\rho_{\Gamma, r} \colon \Gamma \times \Gamma^* 
\hookrightarrow \GL_{nt}/\mu_m. \notag
\end{align}

Note that since we are assuming that $k$ contains a primitive $m$th root of unity, the homomorphism of $\rho_{\Gamma, r} \colon \Gamma \times \Gamma^* \hookrightarrow \GL_{nt}/\mu_m$ of algebraic groups is, in fact, defined over $k$. This means that reduction of structure from $\GL_{nt}/mu_m$ to $H = \Gamma \times \Gamma^*$ can be analyzed using \Cref{thm.reduction-of-structure}.
 
More specifically, let $E[t] \to X$ be the $\GL_{nt}/\mu_m$-torsor over $X$ corresponding to the Azumaya algebra $\mathcal{M}_t(\mathcal{A}) := \mathcal{A} \otimes_X \mathcal{M}_t(\mc{O}_X)$ of degree $nt$ and exponent dividing $m$. Here $\mc{M}_t(\mc{O}_X)$ denotes the split Azumaya algebra of degree $t$ over $X$. 
By \Cref{thm.reduction-of-structure}, there exist a countable collection $\{S_j\}_{j\geqslant 1}$ of closed subschemes  of $S$ such that, for every geometric point $s$ of $S$, the following conditions are equivalent:

\smallskip
(i) $E[t] \to X_s$ descends to $\Gamma \times \Gamma^*$ over every generic point of $X_s$, 

\smallskip
(ii) $s$ lies in $S_j$ for some $j \geqslant 1$. 

\smallskip 
Combining this with the equivalence (\ref{eq:reduction}), we see that the countable collection $\{S_j\}_{j\geqslant 1}$ has the property claimed in the statement of \Cref{prop.symbol-length}.
\end{proof}

\section{Strongly unramified torsors}\label{sect.strongly-unramified}

\begin{defin}\label{def:star}
    Let $k$ be a field, let $G$ be a $k$-group, and let $\mc{P}$ be a property of $G$-torsors over fields. We say that $\mc{P}$ \emph{satisfies condition (*)} if, whenever we are given
    \begin{itemize}
        \item an algebraically closed field $k'$ containing $k$,
        \item a $k'$-scheme of finite type $S$,
        \item a faithfully flat morphism of finite type $X\to S$ such that a general fiber of $X\to S$ is geometrically integral,
        \item a $G$-torsor $\alpha\colon E\to X$, and
        \item a $k'$-point $s_0\in S(k')$ such that $X_{s_0}$ is reduced and $\alpha_{s_0}$ does not satisfy $\mc{P}$ generically,
    \end{itemize}
    the set $\Omega\subset S(k')$ consisting of those $k'$-points $s$ such that $\alpha_s$ does not satisfy $\mc{P}$ generically is very general in $S$.
 \end{defin}   

The relationship among $E$, $X$, $S$, $k'$ and $k$ is shown in the diagram below.

 \[  \xymatrix{  E_{s_0} \ar@{->}[d]^{\alpha_{s_0}} \ar@{^{(}->}[r] &  E \ar@{->}[d]^{\alpha} & \\
 X_{s_0} \ar@{->}[d] \ar@{^{(}->}[r] & X \ar@{->}[d] &  \\
 \Spec(k') \ar@{^{(}->}[r]^{s_0} & S \ar@{->}[d]  & \\
 & \Spec(k')   \ar@{=}[lu] \ar@{->}[r] & \Spec(k). 
 } 
 \]

\begin{rmk} \label{rem.*}
    By \Cref{thm.main2-torsors}, if $\mc{P}$ satisfies the rigidity condition, as well as the specialization condition with respect to all pairs $(\on{Spec}(A),T\to \on{Spec}(A))$, where $A$ is a complete discrete valuation ring which contains $k$ and $T\to \on{Spec}(A)$ is a $G$-torsor, then $\mc{P}$ satisfies condition (*).
\end{rmk}

\begin{example}\label{ex.star}
    Let $k$ be a field, and let $G$ be a $k$-group. The following properties of $G$-torsors over fields satisfy (*).

    \begin{enumerate}
        \item Assume that $G^0$ is reductive, let $P\subset G$ be a subgroup such that $P^0$ is a parabolic subgroup of $G$, and let $\mc{P}$ be the property of $G$-torsors over fields  given by admitting 
        reduction of structure to $P$; see \Cref{defin:reduction-of-structure}. It follows that $\mc{P}$ satisfies the rigidity condition (see \Cref{prop:properties-grs}) and, by the valuative criterion for properness, the specialization condition with respect to $G$-torsors over complete discrete valuation rings. Thus $\mc{P}$ satisfies condition (*).
        \item Let $H\to G$ be a $k$-group homomorphism, and let $\mc{P}$ be the property of $G$-torsors over fields given by admitting 
         reduction of structure to $H$. Assume that $G$ and $H$ are smooth and that $H$ is in good characteristic. Property $\mc{P}$ satisfies the rigidity condition by \Cref{rigidity.reduction-of-structure}, and the specialization condition with respect to torsors over complete discrete valuation rings by \Cref{specialization.reduction-of-structure}. Thus $\mc{P}$ satisfies condition (*).
       \item Assume that $G^0$ is reductive and $G$ is in good characteristic. For a fixed non-negative integer $n$, consider
     property $\mc{P}_n$ from Example~\ref{ex.ed}. Let $k \subset k' \subset K$ be fields, where $k'$ is algebraically closed and let 
    $\alpha \colon E \to \Spec(K)$ be a $G$-torsor.
     Recall that, by definition, $\alpha$ has property $\mc{P}_n$ if $\ed_{k'}(\alpha) \leqslant n$. 
     As we pointed out in Example~\ref{ex.ed}, property~$\mc{P}_n$ satisfies the rigidity condition and the specialization condition with respect to pairs $(\on{Spec}(A),T\to \on{Spec}(A))$, where $A$ is a complete discrete valuation ring which contains $k$ and $T\to \on{Spec}(A)$ is a $G$-torsor. Thus it satisfies condition (*) by Remark~\ref{rem.*}.
       \item If we replace the essential dimension $\ed_{k'}(\alpha)$ by the resolvent degree $\rd_{k'}(\alpha)$, the resulting property also
       satisfies condition (*); see \Cref{ex.rd} for $n \geqslant 1$ and \Cref{ex.split} for $n = 0$.
    \end{enumerate} 
\end{example}

Let $k$ be a field, let $G$ be an affine $k$-group, let $K/k$ be a finitely generated field extension, and let $\tau \colon T \to \Spec(K)$ be a $G$-torsor. We say that $\tau$ is \emph{strongly unramified} if it can be spread out to a $G$-torsor over a projective $k$-variety. In other words, $\tau$ is strongly unramified, if there is a projective variety $Y$ over $k$ with function field $K$, and a $G$-torsor $\tilde{\tau} \colon \tilde{T} \to Y$ such that $\tau$ can be obtained from $\tilde{\tau}$ via the cartesian square 
\[  \xymatrix{  T \ar@{->}[d]_{\tau}   \ar@{->}[r]  &   \tilde{T}  \ar@{->}[d]^{\tilde{\tau}}   \\          
	\Spec(K)   \ar@{->}[r]^{\eta}     &  Y ,   } \]
where $\eta \colon \Spec(K) \to Y$ is the generic point of $Y$.

\begin{prop} \label{prop.combination}
 Let $k$ be an uncountable algebraically closed field, and let $G$ be a smooth affine $k$-group in good characteristic (see Definition~\ref{def.good-characteristic}). Suppose that $\mathcal{P}_1, \ldots, \mathcal{P}_t$ are properties of $G$-torsors over fields such that 
\begin{itemize}
    \item $\mathcal{P}_1, \ldots, \mathcal{P}_t$ satisfy condition (*), and
    \item there exist field extensions $K_1/k, \ldots, K_t/k$ of transcendence degree $e$ and $G$-torsors $\tau_1 \colon T_1 \to \Spec(K_1), \ldots, \tau_t \colon T \to \Spec(K_t)$ such that $\tau_i$ does not have property $\mathcal{P}_i$, for $i = 1, \ldots, t$.
\end{itemize}

\smallskip
Then there exists a finitely generated field extension $K/k$ of transcendence degree $e$ and a strongly unramified
$G$-torsor $\tau_0 \colon T_0 \to \Spec(K_0)$ such that $\tau$ does not have property $\mathcal{P}_i$ for any $i = 1, \ldots, t$.
\end{prop} 

\begin{proof} By our assumption, $G$ is in good characteristic.
This means that there exists a finite discrete subgroup $H \subset G$ such that each of the $G$-torsors $\tau_i$ admits reduction of structure to $H$. (Here $H$ is the subgroup we denoted by $\Lambda$ in \Cref{def.good-characteristic}.) In other words, there exist $H$-torsors $\epsilon_i \colon E_i \to \Spec(K_i)$ such that $\tau_i$ is the image of $\epsilon_i$ under the natural map $H^1(K_i, H) \to H^1(K_i, G)$.
Each $\epsilon_i$ can be spread out to an $H$-torsor 
\begin{equation} \label{e.tilde{E}_i}
\tilde{E}_i \to X_i,
\end{equation} where $X_i$ is an $e$-dimensional integral $k$-variety (not necessarily complete).

By~\cite[Theorem 1.4]{reichstein2023behavior} there exist a smooth  irreducible quasi-projective $k$-variety $S$, a smooth 
		 irreducible quasi-projective $H$-variety $\mc{E}$ and a smooth $H$-equivariant morphism $\pi \colon \mc{E}\to S$ 
		of constant relative dimension $e$ defined over $k$ such that:
		\begin{enumerate}[label=(\roman*)]
			\item $H$ acts trivially on $S$ and freely on $\mc{E}$,
				\item there exists a dense open subscheme $U \subset S$ such that for every $s\in U$ the fiber $\mc{E}_s$ is smooth, 
			projective and geometrically irreducible,
			\item there exists $s_0 \in S(k)$ such that the fiber $\mc{E}_{s_0}$ of $\pi$ over $s_0$
			is $H$-equivariantly birationally isomorphic to the disjoint union of $\tilde{E}_1, \ldots, \tilde{E}_t$.
		\end{enumerate}
		
        In particular, for any geometric point $s$ of $U$, the $G$-action on the fiber $\mc{E}_s$ is strongly unramified.

Since the $G$-action on $\mathcal{E}$ is free, $\mathcal{E}$ may be viewed as an $H$-torsor $\epsilon \colon \mathcal{E} \to \mathcal{X}$ over $\mathcal{X} = \mathcal{E}/H$. Note that in general $\mathcal{X}$ is an algebraic space, but since $H$ is finite and $\mathcal{E}$ is quasi-projective, $\mathcal{X}$ is, in fact, a scheme. Since $H$ acts trivially on $S$, the projection $\pi \colon \mathcal{E} \to S$ factors through $\mathcal{X}$. In other words, there exist a morphism $\overline{\pi} \colon \mathcal{X} \to S$ such that $\pi = \overline{\pi} \circ \epsilon$.

Now let $\tau \colon \mc{T} \to \mathcal{X}$ be the $G$-torsor 
induced by the $H$-torsor $\epsilon \colon \mathcal{E} \to \mathcal{X}$, that is,
\[\mathcal{T} = G *_H \mathcal{X} = (G \times \mathcal{X})/H,\] where $H$ acts on $G \times \mathcal{X}$ by $h \colon (g, x) \mapsto (gh^{-1}, hx)$. 

The relationship among $S$, $\mathcal{X}$, $\mathcal{E}$ and $\mathcal{T}$ is pictured on the left side the diagram below.
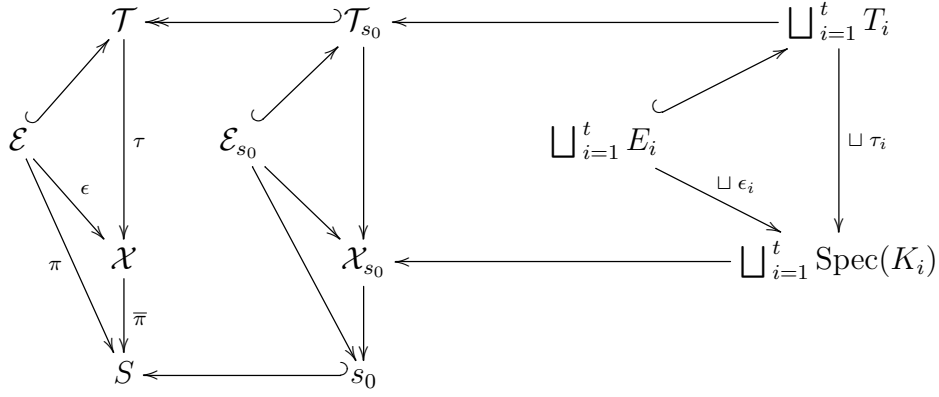
\[  \xymatrix{  & \mathcal{T} \ar@{->}[dd]^{\tau}  &  & \mathcal{T}_{s_0} \ar@{_(->>}[ll] \ar@{->}[dd] &  & & \bigsqcup_{\; i = 1}^{\; t} T_i \ar@{->}[dd]^{\sqcup \; \tau_i} \ar@{->}[lll] \\
\mathcal{E} \ar@{^{(}->}[ur] \ar@{->}[dr]^{\epsilon \; \;}   \ar@{->}[ddr]_{\pi}  &  & \mathcal{E}_{s_0} \ar@{^{(}->}[ur] \ar@{->}[dr]   \ar@{->}[ddr]  &  & & 
\bigsqcup_{\; i = 1}^{\; t} E_i \ar@{^{(}->}[ur] \ar@{->}[dr]^{\sqcup \; \epsilon_i} & \\         	    &  \mathcal{X} \ar@{->}[d]^{\overline{\pi}} & &  \mathcal{X}_{s_0} \ar@{->}[d]  & & & 
\bigsqcup_{\; i = 1}^{\; t} \Spec(K_i)
\ar@{->}[lll] \\ 
                & S   &  & \ar@{_(->}[ll] s_0  & &  & }      
                \] 
The middle part of the diagram shows the fiber over
the geometric point $s_0$ of $S$. By our construction $\mc{E}_{s_0}$
is birationally isomorphic to the disjoint union $\displaystyle \bigsqcup_{\; i = 1}^t \, \tilde{E}_i$,
where $\tilde{E}_i$ is the total space of the $H$-torsor~\eqref{e.tilde{E}_i}. Passing to the generic point $\Spec(K_i) \to X_i$ for each $i$, we obtain the the right side of the diagram.

For each $i = 1, \ldots, t$, since property $\mc{P}_i$ satisfies condition (*), there is a very general subset $\Omega_i\subset S(k)$ such that for every $s\in \Omega_i$ the $G$-torsor $\mathcal{T}_s \to \mathcal{X}_s$ does not satisfy $\mathcal{P}_i$ generically. The finite intersection $\Omega\coloneqq \cap_{i=1}^t\Omega_i$ is also very general, and hence, since $k$ is algebraically closed and uncountable, $\Omega$ is Zariski dense in $S$. Pick $u\in U(k)\cap \Omega$. Note that since
$u \in U(k)$, (ii) tells us that $\mathcal{E}_u$ is geometrically irreducible, and hence so is $\mc{X}_u$. 
The restriction of the $G$-torsor $\mc{T}_u \to X_u$ to $\on{Spec}(k(X_u))$ is strongly unramified and does not satisfy $\mc{P}_i$ for any $i=1,\ldots,t$. This is the torsor $\tau_0$ whose existence is asserted by Proposition~\ref{prop.combination}.
\end{proof} 

\section{Unramified non-crossed product algebras}\label{sect.unramified-non-crossed}

Let $A$ be a central division algebra of degree $n$ over a field $K$ and $\Gamma$ be a finite group of order $n$. Recall that
$A$ is called a $\Gamma$-{\em crossed product} if it has a maximal subfield $K \subset L \subset A$ such that $L$ is $\Gamma$-Galois over $K$.
We will refer to $A$ as a {\em crossed product} if it is a $\Gamma$-crossed product for some finite group $H$ of order $n$, and as a {\em non-crossed product} otherwise. 

For many years it was not known whether or not every finite-dimensional division algebra is a crossed product. Using the primary decomposition theorem, one readily reduces this question to the case, where $n$ is a prime power. The first examples of non-crossed products were constructed by Amitsur~\cite{amitsur72}. He showed that 
the universal division algebra $\UD(p^r, k)$ of degree $p^r$ over $k$
is a non-crossed product for every prime integer $p$, every integer exponent $r \geqslant 3$ and every base field $k$ of characteristic not divisible by $p$.

Because the center of $\UD(n, k)$ is rather large (it is of transcendence degree $n^2 + 1$ over $k$), there has been further work on this topic, constructing non-crossed products over ``small" fields, mostly of an arithmetic nature; see~\cite{brussel95, jacob-wadsworth96, hanke04, hanke05}. In a more geometric direction, the first author and Youssin~\cite{reichstein-youssin2} constructed non-crossed product $A$ of degree $p^r$ ($r \geqslant 3$) over the function field $K$ of 
a $6$-dimensional complex variety. As an application of \Cref{prop.combination}, we will now show that in this example $A$ can be taken to be strongly unramified.

\begin{prop} \label{prop.non-crossed} Let $k$ be an uncountable algebraically closed field of characteristic zero, let $p$ be a prime and let $r \geqslant 3$ be an integer. Then there exists a $6$-dimensional irreducible projective variety $Z$ defined over $k$ and a central division algebra $A$ over $K = k(Z)$ of degree $p^r$, such that

\smallskip
\begin{enumerate}[label=(\roman*)]
\item $A$ is represented by an Azumaya algebra over $Z$, and
\item $A$ is a non-crossed product over $K$. 
\end{enumerate}
\end{prop}

\begin{proof} Let $\Gamma_1, \ldots, \Gamma_t$ be finite groups of order $n := p^r$. For each $i$ we view $\Gamma_i$ as a subgroup of $\PGL_n$, where the embedding $\Gamma_i \hookrightarrow \PGL_n$ is given by
composing the regular representation $\Gamma_i \hookrightarrow \GL_n$ with the natural projection $\GL_n \to \PGL_n$.

Let $D$ be the $n-1$-dimensional diagonal maximal torus of $\PGL_n$ and $H_i$ be the subgroup generated by $D$ and $\Gamma_i$. It is easy to see that $H_i$ is, in fact, a semidirect product $D \rtimes \Gamma_i$. Moreover, a central simple algebra $A$ over a field $K$
is an $\Gamma_i$-crossed product if and only if the $\PGL_n$ torsor
$E \to \Spec(K)$ associated to $A$ admits reduction of structure to $H_i$.

We let $\mathcal{P}_i$ be the property of admitting reduction of structure to $H_i$ for $\PGL_n$-torsors. In other words, a $\PGL_n$-torsor 
$\tau_i \colon E_i \to \Spec(K)$ has property $\mc{P}_i$ if and only if the central simple algebra $A_i/K$ of degree $n$ associated to $\tau$ is an $\Gamma_i$-crossed product.

We will now apply \Cref{prop.combination} to the properties $\mathcal{P}_1, \ldots, \mathcal{P}_t$, with $G = \PGL_{p^r}$ and $e = 6$.
Assumption (i) of \Cref{prop.combination} is satisfied by \Cref{ex.star}(2). (The $H_i$ are in good characteristic because $\on{char}(k)=0$.) It remains to check Assumption (ii) of \Cref{prop.combination}. For this, we refer to
\cite[Theorem 1.4]{reichstein-youssin2}, which asserts that for any prime integer $p$ and any $r \geqslant 3$, there exists a finitely generated extension $K/k$ of transcendence degree $6$ and a non-crossed product central simple algebra $A/K$ of degree $p^r$.\footnote{In \cite{reichstein-youssin2}, Theorem 1.4 is stated for all $r\geqslant 2$, but this is a misprint: the proof is only valid for $r \geqslant 3$. The existence of a non-crossed product of degree $p^2$ over any field containing an algebraically closed field $k$, remains an open problem.} In other words, $A$ is not an $\Gamma_i$-crossed product for any $i = 1, \ldots, t$. Equivalently, the $\PGL_{p^r}$-torsor $\tau \colon E\to \Spec(K)$ associated to $A$ does not have property $\mathcal{P}_i$ for any $i = 1, \ldots, t$. Taking $K_1 = \ldots = K_t = K$ and $\tau_1 = \ldots = \tau_t \colon E \to \Spec(K)$, we see that Assumption (ii) of \Cref{prop.combination} is satisfied, as desired. 
\end{proof}

\begin{rmk} The assumption that $\Char(k) = 0$ comes 
from~\cite[Theorem 1.4]{reichstein-youssin2}, whose proof relies on resolution of singularities. If this assumption can be weakened to $\Char(k) \neq p$ in~\cite[Theorem 1.4]{reichstein-youssin2}, then
it can also be weakened to $\Char(k) \neq p$ 
in Proposition~\ref{prop.non-crossed}.
\end{rmk}

\section{Unramified counterexamples to the period-index problem}\label{sect.period-index}

Let $K$ be a field and $A$ be a central simple algebra of index $\ind(A)$ and period $\per(A)$. By a theorem of Brauer, 
$\per(A)$ divides $\ind(A)$, and they have the same prime factors; see~\cite[Theorem 2.8.7]{gille2017central}.
In general nothing more can be said about the relationship between the period and the index: for any positive integers 
$m$ and $n$ such that $m$ divides $n$ and $m$, $n$ have the same prime factors, there exist a field $K$ and a division algebra $A/K$ of index $n$ and period $m$. However, for certain classes of fields $K$, one can show that $\ind(A)$
divides $\per(A)^e$, where $e$ is a fixed positive integer, 
independent of the choice of the field $K$ (within the specified class) or of the central simple algebra $A/K$.
The question of what the optimal exponent $e$ is for various classes of fields, is known as ``the period-index problem". 
Note that by the primary decomposition theorem, one can always assume that $\ind(A)$ (and thus $\per(A)$) is a prime power.

M.~Artin conjectured that if 
$K$ is a $C_2$-field, then $\per(A) = \ind(A)$. This conjecture was proved by Artin and Tate in the case, where
$\ind(A)$ is a power of $2$ or $3$; see \cite[Appendix]{artin-bs-varieties}. De Jong~\cite{dejong} proved Artin's conjecture for function fields of surfaces over an algebraically closed field. The following conjectural generalization to function fields of higher dimensional varieties first appeared in an unpublished preprint by Colliot-Th\'el\`ene~\cite{colliot2001survey}; see also 
the discussions in~\cite[Section 1]{dejong}.

\begin{conjecture} \label{conj.period-index}
Let $k$ be an algebraically closed field, $d$ be a positive integer, $Z$ be an integral $d$-dimensional $k$-variety and $K = k(Z)$ be the function field of $Z$. For every central simple algebra $A$ over $K$, $\ind(A)$ divides $\per(A)^{d-1}$.
\end{conjecture}

One of the steps in de Jong's proof of Conjecture~\ref{conj.period-index} for $d = 2$ is reduction to the case, where $A$ is strongly unramified, i.e., $A$ comes from an Azuaya algebra over some $d$-dimensional projective variety; see~\cite[Section 6]{dejong}. The following proposition generalizes this step to all $d \geqslant 1$.

\begin{prop} \label{prop.period-index} Assume that the base field $k$ is algebraically closed, of characteristic zero, and uncountable. Suppose Conjecture~\ref{conj.period-index} fails for some $d \geqslant 1$. Then a strongly unramified counterexample with the same $d$ can be found. In other words, 
there exists a counterexample, where $Z$ is projective and $A$ is the generic fiber of an Azumaya algebra $\mathcal{A}$ over $Z$.
\end{prop}

 The statement and proof outline of this proposition first appeared in Section 5 of the unpublished preprint~\cite{dejong-starr} by Starr and de Jong\footnote{As explained to us by de Jong, this preprint was a precursor to the paper~\cite{dejong-starr2010}, but \Cref{prop.period-index} did not make it into~\cite{dejong-starr2010}.}. Here we deduce it as an easy consequence of our \Cref{prop.combination}.

\begin{proof}[Proof of \Cref{prop.period-index}] Assume Conjecture~\ref{conj.period-index} fails for some finitely generated field extension $K_1/k$
of transcendence degree $d$
and a central simple algebra $A_1$ over $K_1$ of degree $n$ and period $m$, where $n$ does not divide $m^{d-1}$.
After replacing $A_1$ by its underlying division algebra
over $K$, we may assume that $A_1$ is a division algebra,
i.e., $\ind(A) = \deg(A) = n$.

Recall that central simple algebras of degree $n$ and exponent dividing $m$ over a field $K$ are in bijective correspondence with $G = \GL_n/\mu_m$-torsors over $K$.
Denote the $G$-torsor corresponding to $A_1$ by 
\begin{equation} \label{e.tau1}
\tau_1  \colon T_1 \to \Spec(K_1).
\end{equation}
We now fix $d$, $n$, $m$ and $G$, and apply \Cref{prop.combination} with $t = 1$ and $e = d$. Let $\mathcal{P} = \mathcal{P}_1$ be following the property 
of a $G$-torsor $\tau \colon X \to \Spec(K)$ (or equivalently, of the central simple algebra $A/K$ associated to $\tau$): $\ind(A) \leqslant n-1$. As we have seen, property $\mc{P}$ for $\tau$ is equivalent to $\tau$ admitting reduction of structure to an appropriate parabolic subgroup of $G$ (cf. \Cref{rem.parabolic}). Thus, by \Cref{ex.star}(1), property $\mc{P}$ satisfies condition (*). Moreover, the $G$-torsor \eqref{e.tau1} does not have property $\mathcal{P}$. By \Cref{prop.combination} we conclude that
there exists a finitely generated field extension $K/k$ of transcendence degree $d$ and a strongly unramified $G$-torsor $T \to \Spec(K)$
which does not satisfy condition $\mathcal{P}$. The central simple algebra $A/K$ associated to this torsor is the counterexample we are seeking. Indeed, $A$ is strongly unramified, $\ind(A) \geqslant n$ (and hence $= n$, since
$\deg(A) = n$), and $\per(A)$ divides $m$. Since $n$ does not divide $m^{d-1}$, we conclude that $n$ does not divide $\per(A)^{d-1}$, as desired.
\end{proof}

\begin{rmk} \label{rem.period-index} In the statement of
Conjecture~\ref{conj.period-index}, we may assume without loss of generality that $\ind(A) = p^r$ and $\per(A) = p^s$
for some prime $p$. If we assume that Conjecture~\ref{conj.period-index} fails in this form (for a specific prime $p$), then the assumption that $\Char(k) = 0$ in \Cref{prop.period-index} can be weakened to
$\Char(k) \neq p$.
\end{rmk}

\begin{rmk} Conjecture~\ref{conj.period-index} remains open for every $d \geqslant 3$. For a recent positive partial result in dimension $3$, see \cite[Theorem 1.3]{hotchkiss2024}. There has also been progress on arithmetic variants of the period-index problem, when $k$ is assumed to be a finite field or a $p$-adic field; see 
\cite{antieau-period-index}.
\end{rmk}

\section*{Acknowledgments} We are grateful to Angelo Vistoli for sharing his unpublished text on rational sections with us. Most of it has made its way, with Angelo's permission, into \Cref{sect.rat-sections} of this paper. We are also grateful to Benjamin Antieau, Johan de Jong and Alena Pirutka for helpful discussions.

\newcommand{\etalchar}[1]{$^{#1}$}

\end{document}